\documentclass[11pt]{article}
\usepackage{amssymb}
\usepackage{amsmath}
\usepackage{graphicx}

\begin{document}

\newtheorem{theorem}{Theorem}[section]
\newtheorem{tha}{Theorem}
\newtheorem{algorithm}[theorem]{Algorithm}
\newtheorem{conjecture}[theorem]{Conjecture}
\newtheorem{corollary}[theorem]{Corollary}
\newtheorem{lemma}[theorem]{Lemma}
\newtheorem{claim}[theorem]{Claim}
\newtheorem{proposition}[theorem]{Proposition}
\newtheorem{construction}[theorem]{Construction}
\newtheorem{definition}[theorem]{Definition}
\newtheorem{question}[theorem]{Question}
\newtheorem{problem}[theorem]{Problem}
\newtheorem{remark}[theorem]{Remark}
\newtheorem{observation}[theorem]{Observation}

\newcommand{\ex}{{\mathrm{ex}}}

\newcommand{\EX}{{\mathrm{EX}}}

\newcommand{\AR}{{\mathrm{AR}}}

\def\endproofbox{\hskip 1.3em\hfill\rule{6pt}{6pt}}
\newenvironment{proof}%
{%
\noindent{\it Proof.}
}%
{%
 \quad\hfill\endproofbox\vspace*{2ex}
}
\def\qed{\hskip 1.3em\hfill\rule{6pt}{6pt}}

\newcommand{\norm}[1]{\lVert#1\rVert}
\def\ce#1{\lceil #1 \rceil}
\def\fl#1{\lfloor #1 \rfloor}
\def\lr{\longrightarrow}
\def\e{\epsilon}
\def\cB{{\cal B}}
\def\cC{{\cal C}}
\def\cD{{\cal D}}
\def\cF{{\cal F}}
\def\cG{{\cal G}}
\def\cH{\mathcal{K}}
\def\cK{{\cal K}}
\def\cI{{\cal I}}
\def\cJ{{\cal J}}
\def\cL{{\cal L}}
\def\cM{{\cal M}}
\def\cP{{\cal P}}
\def\cQ{{\cal Q}}
\def\cS{{\cal S}}
\def\cT{{\cal T}}
\def\imp{\Longrightarrow}
\def\1e{\frac{1}{\e}\log \frac{1}{\e}}
\def\ne{n^{\e}}
\def\rad{ {\rm \, rad}}
\def\equ{\Longleftrightarrow}
\def\pkl{\mathbb{P}^{(k)}_l}
\def\podd{\mathbb{P}^{(k)}_{2t+1}}
\def\peven{\mathbb{P}^{(k)}_{2t+2}}
\def\wt{\widetilde}
\def\wh{\widehat}
\def\Tr{T_r}
\def\tr{t_r}
\def\tx{\tilde{x}}
\def\ty{\tilde{y}}
\def\cmr{\frac{[m]_r}{m^r}}
\def\ckr{\frac{(k+r-3)_r}{(k+r-3)^r}}
\def\ve{\varepsilon}
\def\a{\alpha}
\def\b{\mathbb{B}}
\def\wb{\widetilde{\b}_n}
\def\hd{\hat{\delta}}

\def\band{[\frac{n}{2}-{2\sqrt{n\ln n}}, \, \frac{n}{2}+{2\sqrt{n\ln n}}]}
\def \gap{{2\sqrt{n\ln n}}}

\voffset=-0.5in

\title{Stability and Tur\'an numbers of a class of hypergraphs via Lagrangians}
\author{
Axel Brandt\thanks{Dept. of Mathematics, University of Colorado Denver, Denver, CO, USA, E-mail: axel.brandt@ucdenver.edu.}
 \quad \quad
David Irwin\thanks{Dept. of Mathematics, Ohio State University, Columbus, OH, USA,
E-mail: irwin.315@osu.edu.} \quad \quad 
Tao Jiang\thanks{Dept. of Mathematics, Miami University, Oxford,
OH, USA. E-mail: jiangt@miamioh.edu. Research supported in part by
National Science Foundation grant DMS-1400249.
 \newline\indent
{\it 2010 Mathematics Subject Classifications:}
05C65, 05C35, 05D05.\newline\indent
{\it Key Words}:  Tur\'an number, hypergraph lagrangian, symmetrization, stability, expanded clique, generalized fan.}
}
\date{October 10, 2015}

\maketitle

\begin{abstract}
Given a family of $r$-uniform hypergraphs $\cF$ (or $r$-graphs for brevity), the
Tur\'an number $ex(n,\cF)$ of $\cF$ is the maximum number of edges in an $r$-graph
on $n$ vertices that does not contain any member of $\cF$. A pair $\{u,v\}$ is {\it covered} in a hypergraph $G$ if some edge of $G$ contains $\{u,v\}$.
Given an $r$-graph $F$ and a positive integer $p\geq n(F)$, let $H^F_p$ denote the 
$r$-graph obtained as follows. Label the vertices
of $F$ as $v_1,\ldots, v_{n(F)}$. Add new vertices $v_{n(F)+1},\ldots, v_p$. For
each pair of vertices $v_i,v_j$ not covered in $F$, add a set $B_{i,j}$ of $r-2$ new vertices and the edge $\{v_i,v_j\}\cup B_{i,j}$, where the $B_{i,j}$'s are pairwise
disjoint over all such pairs $\{i,j\}$. We call $H^F_p$ {\it the expanded $p$-clique with an embedded $F$}. For a relatively large family of $F$, we show that for all sufficiently large $n$, $ex(n,H^F_p)=|T_r(n,p-1)|$, where $T_r(n,p-1)$ is the balanced complete $(p-1)$-partite $r$-graph on $n$ vertices. We also establish structural stability of near extremal graphs. 
Our results generalize or strengthen several earlier results and provide
a class of hypergraphs for which the Tur\'an number is exactly determined (for large $n$).

\end{abstract}


\section{Introduction}

Given a family of $r$-uniform hypergraphs $\cF$ (or $r$-graphs for brevity), the
Tur\'an number $ex(n,\cF)$ of $\cF$ is the maximum number of edges in an $r$-graph
on $n$ vertices that does not contain any member of $\cF$. The Tur\'an density $\pi(\cF)$
of $\cF$ is defined to be $\lim_{n\to \infty} ex(n,\cF)/\binom{n}{r}$; such a limit is known to exist. Determining Tur\'an numbers of graphs and hypergraphs is one of the central problems in extremal combinatorics.
For $r=2$, the problem was asymptotically solved for all non-bipartite graphs in form
of the Erd\H{o}s-Stone-Simonovits Theorem that states that if $\cF$ is a family of graphs and the minimum chromatic number among all members is $p\geq 3$ then $\pi(\cF)=\frac{p-2}{p-1}$.  For $r\geq 3$, not too much is known. There are  very few exact or asympotitic results. For a recent 
 on hypergraph Tur\'an numbers, the reader is referred to the survey of Keevash \cite{keevash}. In this paper, we build on earlier works of Sidorenko \cite{sidorenko}, Pikhurko \cite{pikhurko,p}, Mubayi \cite{mubayi},  and Mubayi and Pikhurko
\cite{MP} to obtain a general theorem that determines the exact Tur\'an numbers
of a class of hypergraphs for all sufficiently large $n$. Our main theorems substantially generalize or strengthen several earlier results.
\section{History}

\subsection{Cancellative hypergraphs}
The study of Tur\'an numbers dates back to Mantel's theorem which states
$ex(n,K_3)=\fl{\frac{n}{2}}\cdot \ce{\frac{n}{2}}$. Katona \cite{katona} suggests an extension of the problem to hypergraphs. An $r$-graph $G$ is called {\it cancellative}
if for any three edges $A,B,C$ satisfying $A\cup B=A\cup C$ we have $B=C$.
Equivalently, $G$ is cancellative if it does not contain three distinct members $A,B,C$
such that one contains the symmetric difference of the other two. When $r=2$ the condition is equivalent to saying that $G$ is triangle-free. Katona asked to determine
the largest size of a cancellative $3$-graph on $n$ vertices. The problem was solved by
Bollob\'as \cite{bollobas}, who showed that for all $n$, the largest size of a cancellative $3$-graph on $n$ vertices is the balanced complete $3$-partite $3$-graph on $n$ vertices. Keevash and Mubayi \cite{KM} gave a new proof of Bollob\'as' result and
established stability of near extremal graphs, showing that all cancellative $3$-graphs
on $n$ with close to the maximum number of edges must be structurally close to
the complete balanced $3$-partite $3$-graph.
Bollob\'as \cite{bollobas} conjectured that for all $r\geq 4$, the largest cancellative $r$-graph on $n$ vertices is the balanced complete $r$-partite $r$-graph on $n$ vertices.
This was proved to be true for $r=4$ by Sidorenko \cite{sidorenko}. However Shearer \cite{shearer} gave counterexamples
showing that the conjecture is false for $r>10$.

\subsection{Generalized triangles}

Frankl and F\"uredi \cite{FF-F5, FF} considered a strengthening of cancellative $r$-graphs. For each $r\geq 2$, let $\sum_r$ consist of all $r$-graphs with three edges $D_1,D_2,D_3$
such that $|D_1\cap D_2|=r-1$ and $D_1\triangle D_2\subseteq D_3$, where
$D_1\triangle D_2$ denotes the symmetric difference of $D_1$ and $D_2$. Let
the {\it generalized triangle} $T_k$ be the member of $\sum_r$ with edges
$\{1,\ldots, r\}, \{1,2,\ldots, r-1, r+1\}$, and $\{r,r+1,r+2,\ldots, 2r-1\}$.
For sufficiently large $n$, Frankl and F\"uredi \cite{FF-F5} showed that
$ex(n,\sum_3)=ex(n,T_3)=\fl{\frac{n}{3}}\cdot \fl{\frac{n+1}{3}}\cdot
\fl{\frac{n+2}{3}}$, with the extremal graph being the balanced $3$-partite $3$-graph on $n$ vertices. In \cite{FF},
Frankl and F\"uredi  determined the exact value of $ex(n,\sum_5)$ for all $n$ divisible by $11$ and the exact value of $ex(n,\sum_6)$ for all $n$ divisible by $12$. For these $n$,
the extremal graphs are blow-ups of the unique $(11,5,4)$ and $(12,6,5)$ Steiner systems. Frankl and F\"uredi \cite{FF} conjectured that for all $r\geq 4$, if $n\geq n_0(r)$ is sufficiently large then $ex(n,\sum_r)=ex(n,T_r)$. Pikhurko \cite{pikhurko} proved the conjecture for $r=4$, showing that $ex(n,\sum_4)=ex(n,T_4)=\fl{\frac{n}{4}}\cdot
\fl{\frac{n+1}{4}}\cdot \fl{\frac{n+2}{4}}\cdot \fl{\frac{n+3}{4}}$, with the balanced
complete $4$-partite $4$-graph on $n$ vertices being the unique extremal graph.
Recently, Norin and Yepremyan \cite{NY} proved Frankl and F\"uredi's conjecture 
for $r=5$ and $r=6$.

\subsection{Expanded cliques and generalized fans}

Given a hypergraph $H$ and a pair $\{x,y\}$ of vertices in $H$, we say that
$\{x,y\}$ is {\it covered} in $H$ if some edge in $H$ contains both $x$ and $y$.
Let $T_r(n,\ell)$ denote the complete $\ell$-partite $r$-graph on $n$ vertices where
no two parts differ by more than one in size.
Mubayi \cite{mubayi} considered the Tur\'an problem for the following family of $r$-graphs. For all $p\geq r\geq 2$ let $\cK^r_p$ denote the family of $r$-graphs
$H$ that contains a set $C$ of $p$ vertices such that every pair in $C$ is covered in $H$.
Let $H^r_p$ denote the unique member of $\cK^r_p$ with edge set $\{\{i,j\}\cup B_{i,j}$: $\{i,j\}\in\binom{[p]}{2}\}$, where the $B_{i,j}$'s are pairwise disjoint $(r-2)$-sets outside $[p]$. We call $H^r_p$  the {\it $r$-uniform expanded $p$-clique}.
For all $n,p,r$, Mubayi \cite{mubayi} showed that $ex(n,\cK^r_p)=e(T_r(n,p-1))$
with the unique extremal graph being $T_r(n,p-1)$. Mubayi further established strucutral stability of near extremal $\cK^r_p$-free graphs. Using this stability property, Pikhurko
\cite{pikhurko} later strengthened Mubayi's result to show that $ex(n, H^r_p)=
e(T_r(n,p-1))$ for all sufficiently large $n$.

Mubayi and Pikhurko \cite{MP} considered the Tur\'an problem for so-called generalized fans. Let $Fan^r$ be the $r$-graph compromising $r+1$ edges $e_1,\ldots, e_r,e$
such that $e_i\cap e_j=\{x\}$ for all $i\neq j$, where $x\notin e$, and $|e_i\cap e|=1$
for all $i$. Note that $Fan^2$ is precisely a triangle. Mubayi and Pikhurko showed that
for all $r\geq 3$ and all sufficiently large $n$, $ex(n,Fan^r)=e(T_r(n,r))=\prod_{i=1}^{r}\fl{\frac{n+i-1}{r}}$.

\section{The general problem on $\cH^F_p$ and $H^F_p$}

The problems mentioned in the previous section can be generalized as follows,
as discussed in Keevash \cite{keevash}. Let $r\geq 3$. Let $F$ be an $r$-graph.
Let $p\geq n(F)$. Let $\cH^F_p$ denote the family of $r$-graphs $H$ that contains
a set $C$ of $p$ vertices, called the {\it core}, such that the subgraph of $H$ induced by $C$ contains a copy of $F$ and such that every pair in $C$ is covered in $H$. Let $H^F_p$ be the member of $\cH^F_p$ obtained as follows. We label the vertices
of $F$ as $v_1,\ldots, v_{n(F)}$. Add new vertices $v_{n(F)+1},\ldots, v_p$. 
Let $C=\{v_1,\ldots, v_p\}$. For
each pair of vertices $v_i,v_j\in C$ not covered in $F$, we add a set $B_{i,j}$ of $r-2$ new vertices and the edge $\{v_i,v_j\}\cup B_{i,j}$, where the $B_{i,j}$'s are pairwise
disjoint over all such pairs $\{i,j\}$. We call $H^F_p$ an {\it expanded $p$-clique with
an embedded $F$}. We call $C$ the {\it core} of $H^F_p$.

Using this notation, we can describe the families of graphs considered in the last section as follows. Let $L$ denote the $r$-graph on $r+1$ vertices consisting of two edges
sharing $r-1$ vertices. Then $\cH^L_{r+1}=\sum_r$ and $H^L_{r+1}=T_r$, the generalized triangle. If $F$ is the $r$-uniform empty graph then $\cH^F_p=\cK^r_p$ and $H^F_p=H^r_p$, the $r$-uniform expanded $p$-clique.
Let $e$ denote a single $r$-set, then $H^e_{r+1}=Fan^r$, the $r$-uniform generalized fan. Our main results in this paper determine the exact value of $ex(n,\cH^F_p)$ and 
$ex(n,H^F_p)$ and establish stability of near extremal graphs for a rather wide family
of $F$. Let us also mention that very recently, Hefetz and Keevash \cite{HK} completely
determined $ex(n,H^{M_2}_6)$ for large $n$ (together with stability), where $M_2$ consists of two disjoint triples.


\section{Notations and definitions}

Before introducing our main results, we give some notations and definitions that will be used throughout the paper.
Given a hypergraph $G$ and a set $S$ of vertices, the {\it link graph } of $S$ in $G$,
denoted by $\cL_G(S)$ is the hypergraph with edge set $\{f: f\subseteq V(G)\setminus S, f\cup S\in G\}$. We write $\cL_G(u)$ for $\cL_G(\{u\})$. 
The {\it degree} of $S$ in $G$, denoted by $d_G(S)$, is the number of edges of $G$ that contain $S$, i.e. $d_G(S)=|\cL_G(S)|$. We denote the minimum vertex degree of $G$ by $\delta(G)$.

Let $p\geq 1$ be an integer.
The {\it $p$-shadow} of $G$, denoted by $\partial_p(G)$, is the set of $p$-sets that are contained in edges of $G$, i.e. $\partial_p(G)=\{f: |f|=p, \exists\, e\in G\,  f \subseteq e\}$. Let $m\geq r\geq 1$ be positive integers. Let $[m]_r$  denote the falling factorial $m(m-1)\cdots m(m-r+1)$. 

A hypergraph $G$ {\it covers pairs} if every pair of its vertices is contained in some edge.
If $G$ is a hypergraph and $S$ is a set of vertices in it, then $G[S]$ denotes the subgraph of $G$ induced by $S$.
\section {Hypergraph Lagrangians and Lagrangian density}
In order to describe our results, we need the notion of lagrangians for hypergraphs.
To motivate the notion of hypergraph Lagrangians, we first review the usual hypergraph symmetrization process and some of its properties.
Two vertices $u,v$ in a hypergraph $H$ are {\it nonadjacent} if $\{u,v\}$ is not covered in $H$.
Given a hypergraph $H$ and two nonadajcent vertices $u$ and $v$  in it,
{\it symmetrizing $v$ to $u$} is
the operation that removes all the edges of $H$ containing $v$ and replaces them
with $\{v\cup D: D\in \cL_H(u)\}$. In other words, we make $v$ a clone of $u$.
The following property is implicit in \cite{keevash}. We re-establish it for completeness.

\begin{proposition} \label{H-free}
Let $p,r$ be positive integers, where $p\geq r+1$.
Let $F$ be an $r$-graph with $n(F)\leq p$ and $G$ an $r$-graph that is
$\cH_p^F$-free. Let $u,v$ be two nonadjacent vertices in $G$. Let $G'$ be obtained
from $G$ by symmetrizing $v$ to $u$. Then $G'$ is also $\cH_p^F$-free.
\end{proposition}
\begin{proof}
First note that $u,v$ have codegree $0$ in $G'$.
Suppose for contradiction that $G'$ contains a member $H$ of $\cH_p^F$ with
$C$ being its core. Since $u,v$ have codegree $0$ in $G'$ and
every pair in $C$ is covered in $H\subseteq G'$, $C$ contains at most one of $u$ and $v$.  For each $e\in H$, if $v\notin e$ let $f(e)=e$ and if $v\in e$ let $f(e)=(e\setminus \{v\})\cup \{u\}$. Let $L=\{f(e): e\in H\}$. Then $L\subseteq G$ and $L$ is
a member of $\cH_p^F$ with either $C$ (if $v\notin C$) or $(C\setminus \{v\})\cup \{u\}$ (if $v\in C$) being the core. This contradicts $G$ being $\cH_p^F$-free.
\end{proof}

Given an $r$-graph $G$ and two nonadjacent vertices $u,v$, if $\cL_G(u)=\cL_G(v)$ then
we say that $u$ and $v$ are {\it equivalent}, and write $u\sim v$. Otherwise we say that $u,v$ are {\it non-equivalent}. Note that $\sim$ is
an equivalence relation on $V(G)$. The {\it equivalence class} of a vertex $v$ consists of
all the vertices that are equivalent to $v$.
\begin{algorithm}{\rm (Symmetrization without cleaning)}
{\rm Let $G$ be an $r$-graph. We perform the following as long as $G$ contains two nonadjacent non-equivalent vertices: let $u,v$ be two such vertices where $d(u)\geq d(v)$, we symmetrize each vertex in the equivalence class of $v$ to $u$. We terminate the process when there exists no more nonadjacent non-equivalent pair.
}
\end{algorithm}
Note that the algorithm always terminates since the number of equivalence classes
strictly decreases after each step that can be performed.

As usual, if $V_1,\ldots, V_s$ are disjoint sets of vertices then $\Pi_{i=1}^s V_i=
V_1\times V_2\times\ldots \times V_s=\{(x_1,x_2,\ldots, x_s): \forall i=1,\ldots,s, x_i\in V_i\}$.
We will abuse notation and use $\Pi_{i=1}^s V_i$ to also denote the set of the corresponding unordered $s$-sets.
If $L$ is a hypergraph on $[m]$, then a {\it blowup} of $L$ is a hypergraph $G$
whose vertex set can be partitioned into $V_1,\ldots, V_m$ such that 
$E(G)=\bigcup_{e\in E(L)} \prod_{i\in e} V_i$.
The following proposition follows immediately from the algorithm.
\begin{proposition} \label{symmetrization-properties}
Let $G$ be an $r$-graph and $G^*$ the graph obtained at the end of the symmetrization
process applied to $G$. Then 
\begin{enumerate}
\item $e(G)\leq e(G^*)$.
\item Let $S$ consist of one vertex from each equivalence class of $G^*$ under $\sim$. Then
$G^*[S]$ covers pairs and $G^*$ is a blowup of $G^*[S]$.
\end{enumerate}
\end{proposition}

Let $G$ be an $r$-graph on $[n]$. A {\it weight function}, or {\it weight assignment}, $f$ on $G$ is a mapping from $V(G)$ to $[0, \infty)$. We say that $f$ is a {\it $1$-sum
weight assignment} if $\sum_{v\in V(G)} f(v)=1$. For every edge $e$ in $G$,
define $f(e)=\prod_{v\in e} f(v)$ and call it the {\it weight} of $e$.
We define a polynomial in the variables $\tx=(x_1,\ldots,
x_n)$ by 
$$p_G(\tx)=r!\cdot \sum_{e\in E(G)} \prod_{i\in e} x_i.$$
We define the {\it lagrangian} of $G$ to be 
$$\lambda(G)=\max\left\{ p_G(x): \forall i=1,\ldots, n, x_i\geq 0, \sum_{i=1}^n x_i=1\right\}.$$
Given an $r$-graph $F$, we define the {\it lagrangian density} $\pi_\lambda(F)$ of $F$ to be
$$\pi_\lambda(F)=\max \{\lambda(G): F\not\subseteq G\}.$$
Note that our definition of the lagrangian follows that of Sidorenko \cite{sidorenko} and
differs from the definition given by Keevash \cite{keevash} by a factor of $r!$.
The following proposition follows immediately from the definition of $\pi_\lambda(F)$.
\begin{proposition} \label{blowup}
Let $F$ be an $r$-graph. Let $L$ be an $F$-free $r$-graph. Let $G$ be an $r$-graph
on $[n]$ that is a blowup of $L$. Then $|G|\leq \pi_\lambda(F)\frac{n^r}{r!}$.
\end{proposition}
\begin{proof}
Suppose $V(L)=[s]$ and let $V_1,\ldots, V_s$ be the  partition of $V(G)$
with $V_i$ corresponding to $i$. For each $i\in [s]$,
let $x_i=|V_i|/n$. Let $\tx=(x_1,\ldots, x_n)$. Then $\forall i\in [s], x_i\geq 0$ and $\sum_{i=1}^s x_i=1$.
Since $G$ is a blowup of $L$, we have 
$$|G|=\sum_{e\in L} \prod_{i\in e} |V_i|=n^r \sum_{e\in L} \prod_{i\in e} x_i=\frac{n^r}{r!} \cdot p_{L}(\tx)\leq \frac{n^r}{r!} \pi_\lambda(F),$$
where the last inequality follows from the definition of $\pi_\lambda(F)$ and the fact
that $L$ is $F$-free. 
\end{proof}

We also mention a quick observation given in \cite{keevash}.

\begin{proposition}{\rm \cite{keevash}} \label{density-equivalence}
If $F$ is an $r$-graph that covers pairs, then $\pi(F)=\pi_\lambda(F)$.
\end{proposition}
The notion of hypergraph lagrangians immediately yields the following tight bounds
on $ex(n,\cH_{m+1}^F)$ for certain $r$-graphs $F$. We describe the bounds in
the following theorem, which  is a more specific version of Theorem 3.1 of \cite{keevash}. We give a proof using our language.

\begin{theorem}{\rm \cite{keevash}} \label{asymptotics}
Let $F$ be an $r$-graph with $n(F)\leq m+1$. Suppose that
$\pi_\lambda(F) \leq \frac{[m]_r}{m^r}$. Then for every $n$ we have
$ex(n,\cH_{m+1}^F)\leq \frac{[m]_r}{m^r}\cdot \frac{n^r}{r!}$. Equality holds if $r$ divides $n$. In particular,
$\pi(\cH_{m+1}^F)=\frac{[m]_r}{m^r}$.
\end{theorem}
\begin{proof}
If $L$ is a member of $\cH_{m+1}^F$ with core $C$, then $\partial_2(L)$ contains
an $(m+1)$-clique since every pair in $C$ is covered in $L$.
 Since $\partial_2( T^r_m(n))$ does not contain an $(m+1)$-clique, $L\not\subseteq T^r_m(n)$. Hence $T^r_m(n)$ is $\cH_{m+1}^F$-free and $ex(n,\cH_{m+1}^F)\geq e(T^r_m(n))$.
Since $\lim_{n\to \infty} |T^r_m(n)|/\binom{n}{r}=\frac{[m]_r}{m^r}$, we have
$\pi(\cH_{m+1}^F)\geq \frac{[m]_r}{m^r}$.

Next, let $G$ be an $\cH_{m+1}^F$-free $r$-graph on $[n]$. Let $G^*$ be the final graph obtained at the end of the symmetrization process applied to $G$.
By Proposition \ref{H-free} and Proposition \ref{symmetrization-properties}, 
$G^*$ is $\cH_{m+1}^F$-free and $e(G^*)\geq e(G)$.
Let $S$ consist of one vertex from each equivalence class of $G^*$.
By Proposition \ref{symmetrization-properties}, $G^*[S]$ covers pairs and
$G^*$ is a blowup of $G^*[S]$. If $F\subseteq G^*[S]$, then since $G^*[S]$
covers pairs, $G^*$ (in fact, $G^*[S]$) contains a member of $\cH_{m+1}^F$, a contradiction. Hence $F\not\subseteq G^*[S]$. 

By Lemma \ref{blowup}, we have
$$|G|\leq |G^*|\leq \pi_\lambda(F)\frac{n^r}{r!}\leq 
\frac{[m]_r}{m^r}\cdot \frac{n^r}{r!}.$$
Since this holds for every $\cH_{m+1}^F$-free $G$ on $[n]$, we have
$ex(n,\cH_{m+1}^F)\leq \frac{[m]_r}{m^r}\cdot \frac{n^r}{r!}$.
Note that when $r$ divides $n$, $|T^r_m(n)|=\frac{[m]_r}{m^r}\cdot \frac{n^r}{r!}$.
Hence $ex(n,\cH_{m+1}^F)= \frac{[m]_r}{m^r}\cdot \frac{n^r}{r!}$ in this case.
Finally, a straightforward calculation shows that  $\pi(\cH_{m+1}^F)=\lim_{n\to\infty} ex(n,\cH_{m+1}^F)/\binom{n}{r} \leq \frac{[m]_r}{r!}$. Hence $\pi(\cH_{m+1}^F)=\frac{[m]_r}{m^r}$.
\end{proof}
Let us mention that even though Theorem \ref{asymptotics} immediately establishes the exact value of $ex(n,\cH^F_{m+1})$, establishing the possible stability for $\cH^F_{m+1}$ and establishing the exact value of $ex(n,H^F_{m+1})$ are much more difficult. In fact, the latter two  are the focus of this paper.

\section{Main results}
Our main results involve the determination of the exact value of $ex(n,H^F_{m+1})$ for certain $r$-graphs $F$ for sufficiently large $n$, together with stability of near extremal $H^F_{m+1}$-free graphs. 

\begin{definition}\label{stable}
{\rm
Let $m,r\geq 2$ be positive integers. Let $F$ be an $r$-graph on at most $m+1$ vertices with $\pi_\lambda(F)\leq \cmr$. We say that $\cH_{m+1}^F$ is {\it $m$-stable}
if for every real $\ve>0$ there are a real $\delta_1>0$ and an integer $n_1$ such that
if $G$ is an $\cH_{m+1}^F$-free $r$-graph with $n\geq n_1$ vertices and
more than $(\cmr-\delta_1)\binom{n}{r}$ edges, then $G$ can be made $m$-partite
by deleting at most $\ve n$ vertices.}
\end{definition}

\begin{theorem} \label{stability} {\bf (Stability)}
Let $m,r$ be positive integers. Let $F$ be an $r$-graph that either has at most $m$ vertices or has $m+1$ vertices one of which has degree $1$. If $\pi_\lambda(F)<\frac{[m]_r}{m^r}$, then $\cH_{m+1}^F$ is $m$-stable.
\end{theorem}

\begin{theorem} \label{stable-to-exact} {\bf(Stability to Exactness)}
 Let $F$ be an $r$-graph that either has at most $m$ vertices or has $m+1$ vertices one of which has degree $1$. If $\cH_{m+1}^F$ is $m$-stable, then there exists an integer $n_2$ such that for all $n\geq n_2$, $ex(n,H^F_{m+1})=|\Tr (n,m)|$.
\end{theorem}

Theorem \ref{stability} and Theorem \ref{stable-to-exact} immediately imply

\begin{theorem} \label{proper} {\bf (Main Theorem)}
Let $m,r$ be positive integers.  Let $F$ be an $r$-graph that either has at most $m$ vertices or has $m+1$ vertices one of which has degree $1$. Suppose
either $\pi_\lambda(F)<\frac{[m]_r}{m^r}$ or $\pi_\lambda(F)=\frac{[m]_r}{m^r}$
and $\cH^F_{m+1}$ is $m$-stable.
Then there exists a positive integer $n_3$ such that for all $n\geq n_3$ we have $ex(n,H_{m+1}^F)=|\Tr(n,m)|$.
\end{theorem}

Theorem \ref{stability} and Theorem \ref{proper} immediately imply stability and exact results on expanded cliques
(\cite{mubayi}, \cite{pikhurko}) and on generalized fans (\cite{MP}), since there
$\pi_\lambda (F)=0<\frac{[m]_r}{m^r}$.
By Proposition \ref{density-equivalence}, we have

\begin{corollary}
Let $m,r$ be positive integers.  Let $F$ be an $r$-graph that either has at most $m$ vertices or has $m+1$ vertices one of which has degree $1$. Suppose $F$ covers pairs. Suppose either $\pi(F)<\frac{[m]_r}{m^r}$ or $\pi(F)=\frac{[m]_r}{m^r}$
and $\cH^F_{m+1}$ is $m$-stable.
Then there exists a positive integer $n_3$ such that for all $n\geq n_3$ we have $ex(n,H_{m+1}^F)=|\Tr(n,m)|$.
\end{corollary}
To introduce our next main theorem, we need a definition.
Given a $2$-graph $G$ and an integer $r\geq 2$,
the {\it $(r-2)$-fold enlargement} of $G$ is an $r$-graph $F$ obtained by 
taking an $(r-2)$-set $D$ that is vertex disjoint from $G$ and letting $F=\{e\cup D: e\in G\}$.

Define the following function 
$$f_r(x)=\frac{\prod_{i=1}^{r-1} (x+i-2)}{(x+r-3)^r}.$$
Note that $f_r(x)>0$ on $[0,\infty)$ and $\lim_{x\to \infty} f_r(x)=0$.
Let $M_r$ denote the last (i.e. rightmost) maximum of the function $f_r$ on the
interval $[2,\infty)$. As pointed out in \cite{sidorenko}, $M_r$ is non-decreasing in $r$,
and can be specifically calculated. For instance, $M_2=M_3=2$, $M_4=2+\sqrt{3}$.
Also, we will define $M_1=2$. The well-known Erd\H{o}s-S\'os conjecture says
that if $T$ is a $k$-vertex tree or forest then $ex(n,T)\leq n(k-2)/2$. 
The conjecture has been verified for many families of trees. The conjecture has also been verified when $k$ is  large \cite{AKSS}.
The following theorem was proved by Sidorenko \cite{sidorenko}.

\begin{theorem} \label{sid-main} {\rm \cite{sidorenko}}
Let $r,k\geq 2$ be integers where $k\geq M_r$.
Let $T$ be a tree on $k$ vertices that satisfies Erd\H{o}-S\'os conjecture.
Let $F$ be the $(r-2)$-fold enlargement of $T$. Then 
$$\pi(\cH_{k+r-2}^F)=\pi_\lambda(F)=
\frac{[k+r-3]_r}{(k+r-3)^r}=(k-2)f_r(k).$$ 
\end{theorem}

In fact, Sidorenko's arguments showed that  $ex(n,\cH_{r+k-2}^F)\leq \frac{[k+r-3]_r}{(k+r-3)^r} \frac{n^r}{r!}$,
where equality is attained if $r+k-3$ divides $n$. However, no
 stuctural stability of near extremal families was established and neither was the exact value of $ex(n,H_{r+k-2}^F)$ determined. 
Recall that $H_{k+r-2}^F$ is a specific member of the family $\cH_{k+r-2}^F$.
We strengthen Sidorenko's result by establishing structural stability
of near extremal $\cH_{k+r-2}^F$-free families and then using this stability to 
establish the exact value of $ex(n,H_{k+r-2}^F)$ for all sufficiently large $n$.  The $k=2$ case is trivial. We henceforth assume $k\geq 3$.

\begin{theorem} {\bf (Stability of enlarged trees)} \label{enlarged-tree-stable}
Let $k\geq 3, r\geq 2$ be integers, where $k\geq M_r$.
Let $T$ be a $k$-vertex tree that satisfies the Erd\H{o}s-S\'os conjecture.
Let $F$ be the $(r-2)$-fold enlargement of $T$. 
Then $\cH_{k+r-2}^F$ is $(k+r-3)$-stable.
\end{theorem}

Theorem \ref{enlarged-tree-stable} and Theorem \ref{stable-to-exact} immediately imply

\begin{theorem} {\bf (Exact result on enlarged trees)} \label{enlarged-trees}
Let $k\geq 3, r\geq 2$ be integers, where $k\geq M_r$.
Let $T$ be a $k$-vertex tree that satisfies the Erd\H{o}s-S\'os conjecture.
Let $F$ be the $(r-2)$-fold enlargement of $T$.  
There exists a positive integer $n_4$ such that for all $n\geq n_4$ we have
$ex(n,H_{r+k-2}^F)=|\Tr(n,r+k-3)|$.
\end{theorem}
When $T=K_{1,2}$ and $F$ is the $1$-enlargement of $T$, $H^F_4$ is the
$3$-uniform generalized triangle $T_3$. So Theorem \ref{enlarged-trees} immediately
yields $ex(n,T_3)=|\Tr(n,3)|=\fl{\frac{n}{3}}\cdot \fl{\frac{n+1}{3}}\cdot
\fl{\frac{n+2}{3}}$ for sufficiently large $n$, which was originally proved in \cite{FF-F5}.
To show that $\cH^F_{k+r-2}$ is $(k+r-3)$-stable, we first establish
stability of the lagrangian function for the tree $T$. The stability of the lagrangian of a tree itself maybe  of independent interest, since the larangian function of a $2$-graph $G$ is not always stable.

For the rest of the paper, we prove Theorems \ref{stability}, \ref{stable-to-exact},
and  \ref{enlarged-tree-stable}.


\section{Reduction from $H_{m+1}^F$-free graphs to $\cH_{m+1}^F$-free graphs}

In this short section,  
we establish a quick fact that every $H_{m+1}^F$-free $r$-graph on $[n]$
can be made $\cH_{m+1}^F$-free by removing $O(n^{r-1})$ edges. In particular,
this implies that to establish stability of near extremal $H_{m+1}^F$-free graphs 
it suffices to establish stabiliity of near extremal $\cH_{m+1}^F$-free graphs.
 
We  need the following result of Frankl on the Tur\'an number of a matching.
As is well-known, for sufficiently large $n$, the Tur\'an number $ex(n,M_{s+1})$ of an $r$-uniform matching $M_{s+1}$ of size $s+1$ is $\binom{n}{r}-\binom{n-s}{r}$, as was shown by Erd\H{o}s \cite{erdos-matching}. However, for
our purpose we will use the following slightly weaker but simpler bound that applies to all $n$.

\begin{lemma} {\rm \cite{frankl-matching}} \label{matching}
If $H$ is an $r$-graph on $[n]$ that contains no $(s+1)$-matching, then $|H|\leq s\binom{n}{k-1}$.
\end{lemma}
In fact, Frankl \cite{frankl-matching} showed 
that if $H$ is an $r$-graph that has no $(s+1)$-matching then $|H|\leq s|\partial_{r-1}
(H)|$.
For an integer $s\geq 2$, an {\it $s$-sunflower} with kernel $D$ is a collection of $s$ distinct sets $A_1,\ldots, A_s$ such that $\forall i,j\in [s], i\neq j$, $A_i\cap A_j=D$.
Given an $r$-graph $G$ and a set $D$, define the {\it kernel degree} of $D$ in $G$,
denoted by $d^*_G(D)$ to be 
$$d^*_G(D)=\max\{s: \mbox{$G$ contains an $s$-sunflower with kernel $D$}\}.$$

\begin{lemma} \label{kernel-degree}
Given an $r$-graph $G$ on $[n]$ and integers $p,d>0$, where $d<r$. There exists a subgraph $G'$ of $G$ with $|G'|\geq |G|-p\binom{n}{d}\binom{n}{r-d-1}$ such that for every $d$-set $D$ in $[n]$ if $d_{G'}(D)>0$ then  $d^*_{G'}(D)>p$.
\end{lemma}
\begin{proof}
Starting with $G$, as long as there exists a $d$-set $D$ of vertices such that the degree of $D$ in the remaining graph is nonzero but is
at most $p\binom{n}{r-d-1}$ we remove all the edges containing $D$. Let $G'$ denote the final remaining subgraph of $G$. Then $|G'|\geq |G|-p\binom{n}{r-d-1}\binom{n}{d}$. It is possible that $G'$ is empty. If $G'$ is nonempty, then
for every $d$-set $D$ that has nonzero degree in $\cG'$, we have $|\cL_{G'}(D)|=d_{G'}(D)> p\binom{n}{r-d-1}$. Since $\cL_{G'}(D)$ is an $(r-d)$-graph on $[n]$, by
Lemma \ref{matching}, it contains a $(p+1)$-matching. Hence, $d^*_{G'}(D)>p$.
\end{proof}

\begin{lemma} \label{removal}
Let $p=n(H_{m+1}^F)$. If $G$ is an $H_{m+1}^F$-free graph on $[n]$,
then $G$ contains an $\cH_{m+1}^F$-free subgraph $G'$ with $|G'|\geq
|G|-p\binom{n}{r-3}\binom{n}{2}$. 
In particular, $\pi(H_{m+1}^F)=\pi(\cH_{m+1}^F)$.
\end{lemma}
\begin{proof} 
Let $G$ be the given $H_{m+1}^F$-free graph on $[n]$. By Lemma \ref{kernel-degree},
$G$ contains a subgraph $G'$ with $|G'|\geq |G|-p\binom{n}{r-3}\binom{n}{2}$
such that for every pair $\{a,b\}$ of vertices if $d_{G'}(\{a,b\})>0$ then $d^*_{G'}(\{a,b\})>p$.
We show that $G'$ is $\cH_{m+1}^F$-free.  Suppose for contradiction that 
$G'$ contains a member $H$ of $\cH_{m+1}^F$. Let $C$ denote the core of $L$. Then $H[C]$ contains a copy of $F$.
Let $\{x,y\}$ be any pair in $C$ that is uncovered by $F$. By definition, $\{x,y\}$ is
covered by some edge of $H$ and hence by some edge of $G'$. So $d_{G'}(\{x,y\})\neq 0$ and thus $d^*_{G'}(\{x,y\})> p$. So $G'$ contains a $(p+1)$-sunflower $\cS$ with kernel $\{x,y\}$. 
Since $p=n(H_{m+1}^F)\geq |C|$, we can find an edge $e$ of $\cS$ containing $\{x,y\}$ that intersects $C$ only in $\{x,y\}$. We can continue the process and cover each uncovered pair $\{a,b\}$ in $C$ using an edge that intersects the current partial
copy $H'$ of $H_{m+1}^F$ only in $a$ and $b$. We can do so since $\{a,b\}$ is the 
kernel of a $(p+1)$-sunflower and $H'$ has at most $p$ vertices. 
 Thus we can find a copy of $H_{m+1}^F$ in $G'$, and thus in $G$, contradicting our assumption that $G$ is $H_{m+1}^F$-free. Hence $G'$ is $\cH_{m+1}^F$-free and
$|G|\leq ex(n,\cH_{m+1}^F)+p\binom{n}{r-3}\binom{n}{2}$. 
Since $ex(n,\cH_{m+1}^F)\leq ex(n,H_{m+1}^F)\leq ex(n,\cH_{m+1}^F)+p\binom{n}{r-3}\binom{n}{2}$, we have $\pi(H_{m+1}^F)=\pi(\cH_{m+1}^F)$.
\end{proof}

\section{Stability of near extremal families and proof of Theorem \ref{stability}}

We use Pikhurko's approach \cite{pikhurko} to establish stability of near extremal families.
First, as in \cite{pikhurko} (and in \cite{HK}, \cite{NY}), we modify the usual symmetrization process by adding a cleaning component. In the algorithm, at any stage, when we discuss the equivalence class of a vertex, it refers to the equivalence class under $\sim$ that we defined earlier.
We always automatically readjust equivalence classes after we apply an operation to a graph. Given an $r$-graph $L$ and a real $\a$ with $0<\a\leq 1$, we say that $L$ is
{\it $\a$-dense} if $L$ has minimum degree at least $\a\binom{n(L)-1}{r-1}$.

\begin{algorithm}{\rm (Symmetrization and cleaning with threshold $\a$)} 
\label{symm-clean}
{\rm 

\noindent {\bf Input:} An $r$-graph $G$.

\noindent{\bf Output:} An $r$-graph $G^*$.

\noindent{\bf Initiation:} Let $G_0=H_0=G$. Set $i=0$.

\noindent{\bf Iteration:}  
For each vertex $u$ in $H_i$, let $A_i(u)$ denote the equivalence class that $u$ is in.
If either  $H_i$ is empty or $H_i$ contains no two nonadjacent nonequivalent vertices, then let $G^*=H_i$ and terminate. Otherwise, let $u,v$ be two nonadjacent nonequivalent vertices in $H_i$, where $d_{H_i}(u)\geq d_{H_i}(v)$. We symmetrize each vertex in $A_i(v)$ to $u$.
Let $G_{i+1}$ denote the resulting graph. Note that after the symmetrization, the equivalence classes may change in $G_{i+1}$. 
But they are still well-defined.  If $G_{i+1}$
has minimum degree at least $\a\binom{n(G_{i+1})-1}{r-1}$, i.e. if $G_{i+1}$ is $\a$-dense, then let $H_{i+1}=G_{i+1}$.
Otherwise we let $L=G_{i+1}$ and repeat the following: let $z$ be any vertex of minimum degree in $L$. We redefine $L=L-z$
unless in forming $G_{i+1}$ from $H_i$ we symmetrized the equivalence class of
some vertex $v$ in $H_i$ to some vertex in the equivalence class of $z$ in $H_i$. 
In that case,
we redefine $L=L-v$ instead.  We repeat the process until $L$ becomes either $\a$-dense or empty. Let $H_{i+1}=L$. We call the process of forming $H_{i+1}$ from $G_{i+1}$ ``cleaning''. Let $Z_{i+1}$ denote the set of vertices removed, so that
$H_{i+1}=G_{i+1}-Z_{i+1}$. By our definition, if $H_{i+1}$ is nonempty then it is $\alpha$-dense.
} 
\end{algorithm}

Our main theorem in this section is the following technical theorem.
Since we want the theorem to be as widely applicable as possible, the statements are rather technical.

\begin{theorem} \label{stability-general}
Let $m\geq r\geq 2$ be integers. Let $F$ be an $r$-graph with $\pi_\lambda(F)\leq \cmr$ such that either $n(F)\leq m$ or $n(F)=m+1$ and $F$
contains a vertex of degree $1$. There exists a real $\gamma_0=\gamma_0(m,r)>0$
such that for every positive real $\gamma<\gamma_0$, there exist a real $\delta>0$ and an integer $n_0$ such that the following is true for all $n\geq n_0$. Let $G$ be an $\cH_{m+1}^F$-free $r$-graph on $[n]$ with $|G|>(\cmr -\delta)\binom{n}{r}$. 
Let $G^*$ be the final graph produced by Algorithm \ref{symm-clean} with threshold $\cmr-\gamma$. Then $n(G^*)\geq (1-\gamma)n$ and $G^*$ is $(\cmr-\gamma)$-dense. Further, if there is a set $W\subseteq V(G^*)$ with $|W|\geq (1-\gamma_0)|V(G^*)|$ such that $W$ is the union of a collection of at most $m$ equivalence classes of $G^*$, then $G[W]$ is $m$-partite.
\end{theorem}

To prove the theorem, we first need to develop a series of lemmas.
First let us mention a routine fact, which is established in \cite{mubayi} and
can be verified straightforwardly.
\begin{lemma} \label{balanced-parts} {\rm (Claim 1 in \cite{mubayi})}
For any integers $m\geq r\geq 2$ and real $\gamma>0$ there exist a real $\beta=\beta(\ve)>0$
and an integer $M_1$ such that any $m$-partite $r$-graph $G$ of order $n\geq M_1$ and
size at least $(\cmr-\beta)\binom{n}{r}$ the number of vertices in each part is
between $\frac{n}{m}-\ve n$ and $\frac{n}{m}+\ve n$.
\end{lemma}

We may assume that $\ve$ is sufficiently small.  
First, we choose small positive reals
$$1\gg c_2\gg c_1\gg \gamma_0 >0,$$
and an integer $n_1$. Our first condition on $n_1$ is that $n_1\geq M_1$, where $M_1$ is given in 
Lemma \ref{balanced-parts}, and that $n_1$ satisfies \eqref{c1-choice} given below. Other conditions on $n_1$ will be stated implicitly throughout the proofs.
We now describe the conditions on the constants as follows.
First we choose $c_1$ to be small enough and $n_1$ large enough so that for all
$N\geq n_1$, we have
\begin{equation}\label{c1-choice}
\left(\cmr-c_1\right)\binom{N-1}{r-1}\geq \left(\cmr-2c_1\right)\frac{N^{r-1}}{(r-1)!}
=\binom{m-1}{r-1}\left(\frac{N}{m}\right)^{r-1}-2c_1\frac{N^{r-1}}{(r-1)!}.
\end{equation}
Next, subject to \eqref{c1-choice}, we choose $c_1,c_2$ to be small enough and $n_1$ large enough so that for $N\geq n_1$,
\begin{equation} \label{constant-choice1}
\left(\cmr-c_1\right)\binom{N-1}{r-1}-\binom{m-2}{r-1}\left(\frac{N}{m}+c_2N\right)^{r-1}>\frac{1}{2}\binom{m-2}{r-2} \left(\frac{N}{m}\right)^{r-1}
>\frac{1}{2m^{r-1}} N^{r-1}.
\end{equation} 
Such choices exist by \eqref{c1-choice} and the fact that
$\binom{m-1}{r-1}-\binom{m-2}{r-1}=\binom{m-2}{r-2}$. In addition, we can
make our choice of $c_1,c_2$ solely dependent on $m$ and $r$.
Now, subject to \eqref{c1-choice} and \eqref{constant-choice1}, we choose $c_1,c_2,\gamma_0$ to
satisfy
\begin{eqnarray} \label{constant-choices}
 c_2<\frac{1}{6m}\left(\frac{1}{10m}\right)^{m-1}, \quad \quad c_2<\frac{1}{(2m)^{mr}}, \quad \quad  c_1<\min\{\frac{c_2}{8}, \beta(\frac{c_2}{2m})\}, \quad
\quad \gamma_0+4\gamma_0(r-1)<c_1, 
\end{eqnarray}
where the function $\beta$ is defined as in Lemma \ref{balanced-parts}.
Note that all $c_1,c_2,\gamma_0$ can be defined solely dependent on $m$ and $r$.

Now, let $\gamma<\gamma_0$ be given. Choose $\delta>0$ to be small enough so that
\begin{equation} \label{gamma-delta-relation}
\frac{\gamma-\delta}{\gamma+\delta}\geq 1-\gamma.
\end{equation}
Let $n_0=2n_1$.
Let $G$ be a $\cH_{m+1}^F$-free graph on $[n]$, where $n\geq n_0$, such that
$$|G|>(\cmr-\delta)\binom{n}{r}.$$

Let $G^*$ be the final graph obtained by applying Algorithm \ref{symm-clean} to $G$ with threshold $\cmr-\gamma$. Suppose the algorithm terminates after $s$ steps.
So, $G^*=G_s$. 

\begin{lemma} \label{removed-vertices}
Let $Z^*=\bigcup_{i=1}^s Z_i$, i.e. $Z^*$ is the set of removed vertices by Algorithm
\ref{symm-clean} with threshold $\cmr-\gamma$. Then $|Z^*|<\gamma n$. Hence, $n(G^*)\geq (1-\gamma)n$ and $G^*$ is $(\cmr-\gamma)$-dense.
\end{lemma}
\begin{proof}
Let $p=|Z^*|$. Let $\alpha=\cmr$.
By the algorithm, when symmetrizing the number of edges doesn't
decrease. When deleting a vertex, the number of edges we lose is at most $(\a-\gamma)\binom{x-1}{r-1}$, where $x$ is the number of vertices remaining in the graph before the deletion of that vertex. Hence
\begin{eqnarray*}
|G_s|&\geq& |G|-(\a-\gamma)\sum_{i=1}^p \binom{n-i}{r-1}\\
&\geq& (\a-\delta)\binom{n}{r}-(\a-\gamma)\left[\binom{n}{r}-\binom{n-p}{r}\right].\\
\end{eqnarray*}
Since symmetrizing preserves $\cH_{m+1}^F$-freeness and deletion of vertices certainly also does, $G_s$ is $\cH_{m+1}^F$-free. By Theorem \ref{asymptotics}, $|G_s|\leq \a\frac{(n-p)^r}{r!}<(\a+\delta)\binom{n-p}{r}$, for sufficiently large $n$.
Hence we have
$$(\a+\delta)\binom{n-p}{r} \geq  (\a-\delta)\binom{n}{r}-(\a-\gamma)\left[\binom{n}{r}-\binom{n-p}{r}\right].$$
This yields
$$(\gamma+\delta)\binom{n-p}{r}\geq (\gamma-\delta)\binom{n}{r}.$$
Hence $$\left(\frac{n-p}{n}\right)^r\geq \frac{\binom{n-p}{r}}{\binom{n}{r}}\geq \frac{\gamma-\delta}{\gamma+\delta}\geq 1-\gamma,$$
where the last inequality holds by \eqref{gamma-delta-relation}. Hence
$$1-\frac{p}{n}\geq \left(1-\gamma\right)^\frac{1}{r}\geq 1-\gamma.$$
So $p\leq \gamma n$. Hence $n(G^*)\geq (1-\gamma)n$. Since the algorithm terminates with a nonempty $G^*$, $G^*$ is $(\cmr-\gamma)$-dense.
\end{proof}

Suppose now that  there exists $W\subseteq V(G^*)$
with $|W|\geq (1-\gamma_0)|V(G^*)|$ such that $W$ is the union of at most $m$ equivalence classes of $G^*$.  
Let $N=|W|$. Then 
\begin{equation} \label{N-lower}
N\geq (1-\gamma_0)(1-\gamma)n\geq n-2\gamma_0 n.
\end{equation}
Since $n\geq n_0$, certainly $N\geq n/2\geq n_1$.

\begin{lemma} \label{min-degree}
For each $i\in [s]$, we have $\delta(G_i[W])=\delta(H_i[W])\geq (\cmr-c_1)\binom{N-1}{r-1}$. Hence, in particular $|G_i[W]|=|H_i[W]|\geq (\cmr-c_1)\binom{N}{r}$.
\end{lemma}
\begin{proof}
Note that $\forall i\in [s]$, $G_i[W]=H_i[W]$, since $H_i=G_i-Z_i$ and $Z_i\subseteq Z^*\subseteq [n]\setminus W$.  For convenience, let $\a=\cmr$.
Let $i\in [s]$. By the algorithm,
$H_i$ is $(\a-\gamma)$-dense, that is, $\delta(H_i)\geq (\a-\gamma)\binom{n(H_i)-1}{r-1}$. For each vertex $x$ in $W$, by \eqref{N-lower}, there are at most $2\gamma_0 n\binom{n(H_i)-2}{r-2}$ edges of $H_i$ that contain $x$ and a vertex outside $W$. For each $i\in [s]$, since $n(H_i)\geq |W|>(1-2\gamma_0)n$,
we have $n\leq \frac{1}{1-2\gamma_0}n(H_i)\leq 2(n(H_i)-1)$. We have
\begin{eqnarray*}
\delta(H_i[W])&\geq& (\a-\gamma)\binom{n(H_i)-1}{r-1}-2\gamma_0 n\binom{n(H_i)-2}{r-2}\\
&\geq&(\a-\gamma_0)\binom{n(H_i)-1}{r-1}-4\gamma_0 (n(H_i)-1)\binom{n(H_i)-2}{r-2}\\
&=& (\a-\gamma_0-4\gamma_0(r-1)) \binom{n(H_i)-1}{r-1}\geq (\a-c_1)\binom{N-1}{r-1},
\end{eqnarray*}
where the last inequality follows from \eqref{constant-choices}.
\end{proof}

Next, we develop a routine but useful lemma on near complete $m$-partite $r$-graphs.
Given an $m$-partite $r$-graph $L$ with parts $A_1,\ldots, A_m$, a {\it transveral}
is a set $S$ of vertices consisting of one vertex from each part.  The transveral $S$ is
{\it complete} if it induces a complete $r$-graph on $S$. A transversal that is
not complete is called {\it noncomplete}.

\begin{lemma} \label{near-complete}
Let $L$ be an $m$-partite $r$-graph on $N\geq n_0$ vertices, where $\delta(L)\geq (\cmr-c_1)\binom{N-1}{r-1}$. Let $A_1,\ldots, A_m$ be an $m$-partition of $L$.
Then
\begin{enumerate}
\item For each $j\in [m], ||A_j|-\frac{N}{m}|<c_2 N$.
\item The number of noncomplete transversals is at most $c_2 N^m$.
\item The number of noncomplete transversals containing any one vertex is at most
$c_2 N^{m-1}$.
\end{enumerate}
\end{lemma}
\begin{proof} Since $\delta(L)\geq (\cmr-c_1)\binom{N-1}{r-1}$, we have
$|L|\geq (\cmr-c_1)\binom{N}{r}$. Since $L$ is $m$-partite on $N\geq M_1$ vertices
and $c_1<\beta(\frac{c_2}{2m})$, by Lemma \ref{balanced-parts},
\begin{equation} \label{balanced}
\forall j\in [m],  \left ||A_j|-\frac{N}{m}\right |<(c_2/2m)N<c_2N.
\end{equation}
Hence item 1 holds.
Let $K$ denote the complete $m$-partite $r$-graph with parts $A_1,\ldots, A_m$.
Then $|K|\leq |\Tr(N,m)|\leq (\cmr+c_1)\binom{N}{r}$, for sufficiently large $N$. Since $|L|\geq (\cmr-c_1)\binom{N}{r}$, we have
$$|K\setminus L|<2c_1\binom{N}{r}.$$
Each noncomplete transversal must contain a member of $K\setminus L$.
On the other hand, for a fixed member of $K\setminus L$, there are at most
$(\max_j|A_j|)^{m-r}\leq (2N/m)^{m-r}$ transversals that contain it.
So the number of noncomplete transversals is at most
$$2c_1\binom{N}{r}(2N/m)^{m-r}<2c_1 N^m\leq c_2 N^m.$$
This proves item 2. It remains to prove item 3.
Let $x$ be any vertex. Without loss of generality, suppose $x\in A_1$.  Let $L_x$ denote the link graph of $x$ in $L$. Let $K_x$ denote the complete $(m-1)$-partite $(r-1)$-graph with parts $A_2,\ldots, A_m$. A complete $(m-1)$-partite $(r-1)$-graph $K'$ with $\fl{\frac{N}{m}}$ vertices in each part has at most $\cmr\binom{N-1}{r-1}$ edges.
Since $||A_j|-\frac{N}{m}|<(c_2/2m)N$ for $j=2,\ldots,m$, we can delete at most $(c_2/2)N$ vertices from $K_x$ to obtain a subgraph of $K'$. Hence, $|K_x|\leq |K'|+ (c_2/2)N^{r-1}<\cmr\binom{N-1}{r-1}+(c_2/2)N^{r-1}$.
Since $|L_x|\geq (\cmr-c_1)\binom{N-1}{r-1}$, we have
$$|K_x\setminus L_x|<(c_1+c_2/2) N^{r-1}<(3c_2/4)N^{r-1}.$$

Let $T$ denote the collection of noncomplete transversals that contain $x$.
Every member of $T$ must either contains an edge $e\in K\setminus L$ where
$x\notin e$ or an edge $\{x\}\cup f\in K\setminus L$ where $f\in K_x\setminus L_x$. The number of
members of $T$ of the former type is at most
$$|K\setminus L|\cdot (\max_j |A_j|)^{m-1-r}\leq 2c_1\binom{N}{r}(2N/m)^{m-1-r}<2c_1N^{m-1}\leq (c_2/4) N^{m-1}.$$
The number of members of $T_2$ of the latter type is at most
$$|K_x\setminus L_x|\cdot (\max_j |A_j|)^{m-r}\leq  (3c_2/4)N^{r-1} (2N/m)^{m-r}<(3c_2/4)N^{m-1}.$$
Hence $|T|\leq c_2N^{m-1}$.
\end{proof}
 
\begin{lemma}\label{Kplus}
If $L$ is an $r$-graph obtained from the complete $r$-graph $K$ on $[m]$ by duplicating vertex $1$ into $1'$ and adding an edge $e$ covering $\{1,1'\}$, then $L$ contains a member of $\cH^F_{m+1}$.
\end{lemma}
\begin{proof} Let $C=[m]\cup \{1'\}$.
Whether $n(F)\leq m$ or $n(F)=m+1$ and $F$ contains a vertex of degree $1$ it is easy to see that $L[C]$ contains $F$ and that all pairs in $C$ are covered in $L$. So $L$ contains  a member of $\cH^F_{m+1}$.
\end{proof}

\medskip

{\bf Proof of Theorem \ref{stability-general}:}
We have already shown that $n(G^*)\geq (1-\gamma)n$ and that $G^*$ is
$(\cmr-\gamma)$-dense. By Lemma \ref{H-free},
symmetrizing preserves $\cH_{m+1}^F$-freeness. Deletion of vertices certainly also does. So $G^*$ is $\cH_{m+1}^F$-free. Next, we want to prove that $G[W]$ is $m$-partite. To do that, 
we use reverse induction on $i$ to prove that $\forall \in [s], G_i[W]$ is $m$-partite.
By our assumption, $G_s[W]$ is $m$-partite. This establishes the basis step.
Let $i< s$. Assume that $G_{i+1}[W]$ is $m$-partite, we prove that $G_i[W]$ must also be $m$-partite. As before, let $N=|W|$.
Let $A^{i+1}_1,\ldots, A^{i+1}_m$ be an $m$-partition of $G_{i+1}[W]$.
By Lemma \ref{min-degree}, we have
$$\delta(G_{i+1}[W])=\delta(H_{i+1}[W])\geq (\cmr-c_1)\binom{N-1}{r-1}
\mbox{ and } |G_{i+1}[W]|=|H_{i+1}[W]|\geq (\cmr-c_1)\binom{N}{r}.$$
By Lemma \ref{near-complete},
$$\forall j\in [m], \left ||A^{i+1}_j|-\frac{N}{m}\right|<c_2N.$$
In particular, we may assume that $\forall j\in [m], N/2m\leq |A^{i+1}_j|\leq 2N/m$.
Let $K_{i+1}$ denote the complete $m$-partite graph on $W$ with parts
$A^{i+1}_1,\ldots, A^{i+1}_m$.

Suppose that in forming $G_{i+1}$ from $H_i$ we symmetrized the equivalence class $C_v$ of $v$ in $H_i$ to some vertex $u$ in $H_i$. If none of $C_v$ is in $W$, then $G_i[W]=G_{i+1}[W]$ and there is nothing to prove. So we may assume that $C_v\cap W\neq \emptyset$. Since all the vertices in $C_v$ are the same, we assume that $v\in C_v\cap W$. By our algorithm this means $u\in W$ as well.
Indeed, since we symmetrized $C_v$ to $u$, by rule in the subsequent cleaning steps
$u$ would be removed only if  all of $C_v$ is removed. Also, from step $i+1$ forward,
$u$ and $v$ always lie in the same equivalence class. Since $W$ is the union of
equivalence classes of $G_s$ and $v\in W$, we should have $u\in W$ as well.
Without loss of generality, suppose $u\in A^{i+1}_1$. Let 
$U_1=A_1^{i+1}\setminus C_v$ and $W'=W\setminus C_v$.
For each $j=2,\ldots, m$, let $U_j=A^{i+1}_j$.
Then $U_1,\ldots, U_m$ is an $m$-partition of $G_{i+1}[W']$ and also note that $G_{i+1}[W']=H_i[W']=G_i[W']$. Let $E_v$ be the set of edges of $H_i[W]$ that
contains $v$. Let $E'_v=\{e\setminus \{v\}: e\in E_v\}$. Then $|E'_v|=|E_v|$.
By Lemma \ref{min-degree},
\begin{equation} \label{Ev-lower}
|E_v|=|E'_v|\geq (\cmr-c_1)\binom{N-1}{r-1}.
\end{equation}

\medskip

{\bf Claim 1.} $\forall e\in E_v, \forall j\in [m]$, we have $|e\cap U_j|\leq 1$.

\medskip

{\it Proof of Claim 1.}
First we show that $\forall e\in E_v$, $|e\cap U_1|\leq 1$. Suppose for
contradiction that there exists $e\in E_v$ with $|e\cap U_1|\geq 2$.
Let $a,b\in e\cap U_1$. Let $\cS$ be the collection of
all $(m-1)$-sets $S$ obtained by selecting one
vertex from  $U_\ell$ for each $\ell\in [m]\setminus \{1\}$. Then 
\begin{equation} \label{S-lower}
|\cS|\geq (N/2m)^{m-1}>2c_2N^{m-1},
\end{equation}
where the last inequality follows from \eqref{constant-choices}.
For each $S\in \cS$, note that $S\cup \{a\}$ and $S\cup \{b\}$ are both transversals
in $G_{i+1}[W]$ (relative to $A^{i+1}_1,\ldots, A^{i+1}_m$). By Lemma \ref{near-complete} there are at most $2c_2N^{m-1}$ noncomplete transversals in $G_{i+1}[W]$  containing either $a$ or $b$.  Hence, by \eqref{S-lower} there exists $S\in \cS$ such that 
$S_1=S\cup \{a\}$ and $S_2=S\cup \{b\}$ are complete transversals in $G_{i+1}[W]$.
That is, $S_1$ and $S_2$ both induce complete $r$-graphs in $G_{i+1}[W]$.
Since $S_1,S_2\subseteq W'$ and $G_i[W']=H_i[W']=G_{i+1}[W']$, $S_1$ and $S_2$
both induce complete $r$-graphs in $H_i[W]$ as well. By Lemma \ref{Kplus}, the union of these two complete
$r$-graphs plus $e$ contains a member of $\cH_{m+1}^F$ in $H_i[W]$,
a contradiction. Hence $\forall e\in E_v, |e\cap U_1|\leq 1$.

Next, let $j\in [m]\setminus \{1\}$. Suppose there exists $e\in E_v$ such that
$|e\cap U_j|\geq 2$. If $|U_1|\geq N/10m$, then we argue just like above
with the only difference being to replace \eqref{S-lower} with $|\cS|\geq (N/10m)^{m-1}>2c_2N^{m-1}$, which still holds by \eqref{constant-choices}.
Hence, we may assume that $|U_1|<N/10m$.

Since $A^{i+1}_1=U_1\cup (C_v\cap W)$ and $|A^{i+1}_1|\geq N/2m$ we have
$|C_v\cap W|\geq 0.4(N/m)$.  Let $a,b\in e\cap U_j$.
Recall that $u\in U_1$. Suppose first the number of
noncomplete transversals in $G_{i+1}[W]$ that contain both $u$ and $a$
is at least $3mc_2N^{m-2}$. Then since all of $C_v\cap W$ is symmetrized to $u$ in
forming $G_{i+1}$ from $H_i$
and $C_v\cap W$ and $u$ are both in $A^{i+1}_1$, the number of noncomplete
transversals in $G^{i+1}[W]$ that contain $a$ is at least 
$$3mc_2N^{m-2}|C_v\cap W|\geq 3mc_2\frac{0.4}{m} N^{m-1}>c_2N^{m-1},$$
contradicting Lemma \ref{near-complete}. Hence, the number of noncomplete transversals containing both $u$ and $a$ is at most $3mc_2N^{m-2}$.
Similarly the number of noncomplete transversals containing both $u$ and $b$
is at most $3mc_2 N^{m-2}$. Let $\cS$ be the collection of $(m-2)$-sets $S$
obtaining by selecting one vertex from $U_j$ for each $j\in [m]\setminus \{1,j\}$.
Then 
\begin{equation} \label{S-lower2}
|\cS|\geq (N/2m)^{m-2} > 6mc_2N^{m-2},
\end{equation}
where the last inequality follows from \eqref{constant-choices}.
For each $S\in \cS$, $S_1=S\cup \{u,a\}$ is a transversal in $G_{i+1}[W]$ containing
both $u$ and $a$ and $S_2=S\cup \{u,b\}$ is a  transversal in $G_{i+1}[W]$ containing both $u$ and $b$. By \eqref{S-lower2}, there exists $S\in \cS$ such that
both $S_1$ and $S_2$ are complete transversal in $G_{i+1}[W]$.
As before they both induce complete $r$-graphs in $H_i[W]$ as well.
Their union together with $e$ now contains a member of $\cH_{m+1}^F$, a contradiction. Hence $\forall e\in E_v, j\in [m], |e\cap U_j|\leq 1$. \qed

\medskip
By Claim 1, for all $f\in E'_v$ and for all $j\in [m]$, $|f\cap U_j|\leq 1$.
So each member $f$ of $E'_v$ intersects some $r-1$ parts among $U_1,\ldots, U_m$.
By an averaging argument, there exist some $r-1$ parts $U_{j_1},\ldots,
U_{j_{r-1}}$ such that at least $|E'_v|/\binom{m}{r-1}$ members of $E'_v$
intersect these $r-1$ parts and no other parts. Let $J=\{j_1,\ldots, j_{r-1}\}$.
Let 
\begin{equation} \label{EJ-lower}
E_J=\{f\in E'_v: \forall j\in J, f\cap U_j\neq \emptyset\}.
\end{equation}
By our discussion,
\begin{equation} \label{E*-lower}
|E_J|\geq |E'_v|/\binom{m}{r-1}\geq (\cmr-c_1)\binom{N-1}{r-1}/\binom{m}{r-1}
>\frac{1}{(2m)^r} N^{r-1},
\end{equation}
for sufficiently large $N\geq n_1$.
\medskip

Let $$I=\{i\in [m]: |\partial_1(E'_v)\cap U_i|\geq \frac{1}{(2m)^r} N\}.$$
By \eqref{E*-lower} and the definition of $I$, we have $J\subseteq I$.
First, suppose that $|I|\leq m-2$.
By our earlier discussion, $\forall i\in I\subseteq [m], |U_i|\leq \frac{N}{m}+c_2N$.
Also, for each $i\notin I$, the number of members of $E'_v$ intersecting $U_i$ is
trivially at most $\frac{1}{(2m)^r}N\cdot N^{r-2}=\frac{1}{(2m)^r}N^{r-1}$. Hence,
by \eqref{constant-choice1} (with room to spare),
$$|E'_v|\leq \binom{m-2}{r-1}\left(\frac{N}{m}+c_2N\right)^{r-1}+m\left[\frac{1}{(2m)^r}N^{r-1}\right]<
(\cmr-c_1)\binom{N-1}{r-1},$$
contradicting \eqref{Ev-lower}.
Hence $$|I|\geq m-1.$$

If $|I|=m-1$, then let $k\in [m]\setminus I$. If $|I|=m$, then let $k\in I\setminus J$.

\medskip

{\bf Claim 2.} We have $\forall e\in E_v$, $e\cap U_k=\emptyset$.

\medskip

{\it Proof of Claim 2.}
Suppose for
contradiction that there exists $e\in E_v$ that contains a vertex $y\in U_k$.
Let $\cT$ be the collection of $(m-r)$-sets $T$ obtained by selecting one vertex from
$\partial_1(E'_v)\cap U_\ell$ for each $\ell\in [m]\setminus (J\cup \{k\})\subseteq I$.
For each $T\in \cT$ and $f\in E_J$, $T'=T\cup f\cup \{y\}$ is a transversal in $G_{i+1}[W]$ containing $y$. The number of different $T'$ is at least
$$|E_J|\left[\frac{1}{(2m)^r}N\right]^{m-r}\geq  \frac{N^{r-1}}{(2m)^r}\cdot
\left[\frac{N}{(2m)^r}\right]^{m-r}=\frac{N^{r-1}}{(2m)^{mr}}> c_2N^{m-1},$$
where the last inequality follows from \eqref{constant-choices}. 
By Lemma \ref{near-complete}, the number
of noncomplete transversals in $G_{i+1}[W]$ containing $y$ is less than $c_2N^{m-1}$.
So there exist $T\in \cT, f\in E_J$ such that $T'=T\cup f\cup \{y\}$ is a complete transversal in $G_{i+1}[W]$. As before, $T'$ also induces a complete $r$-graph in $H_i[W]$. Now we can find a member of $\cH_{m+1}^F$ in $H_i[W]$ as follows. Let $C'=\{v\}\cup T'$. If $n(F)\leq m$, then we map $F$ into $T'$. If $n(F)=m+1$ and
$z$ is a degree $1$ vertex in $F$, then we map $F$ into $C'$ with $z$ mapped to $v$.
Such mappings exist since $v\cup f\in H_i[W]$ and $T'$ is complete in $H_i[W]$.
It remains  to check that all pairs in $C'$ are covered in  $H_i[W]$. Pairs not containing $v$ are covered since $H_i[T']$ is complete. Pairs of the form $\{v,a\}$ where $a\in f$ are covered by $\{v\}\cup f$. The pair $\{v,y\}$ is covered by $e$. The remaining pairs have the form $\{v, b\}$, where $b\in T$. By our definition of $T$, $b\in \partial_1(E'_v)$.
Hence there exist an edge in $E_v$ that contains $v$ and $b$. We have thus shown that
$H_i[W]$ contains a member of $\cH_{m+1}^F$.
This contradicts $H_i[W]$ being $\cH_{m+1}^F$-free. \qed

\medskip

We have thus far shown that each edge in $E_v$ intersects each $U_j$ in at most one vertex and intersects $U_k$ in no vertex.  Since all the vertices in $S_v\cap W$ behave the same as $v$ in $H_i[W]$, we see that $H_i[W]$ is $m$-partite with an $m$-partition
$U'_1,\ldots, U'_m$, where $U'_j=U_j$ for each $j\in [m]\setminus \{k\}$ and
$U'_k=U_k\cup (S_v\cap W)$.
This completes the induction and the proof of Theorem \ref{stability-general}.
\qed

\bigskip

Now, we can prove Theorem \ref{stability}, namely we show that if 
$F$ is an $r$-graph with $\pi_\lambda(F)<\frac{[m]_r}{m^r}$ such that
either $n(F)\leq m$ or $n(F)=m+1$ 
and $F$ contains a vertex of degree $1$, then $\cH_{m+1}^F$ is $m$-stable.

\medskip

{\bf Proof of Theorem \ref{stability}:}
Let $\ve>0$ be given. We may assume that $\ve$ is sufficiently small so that
$\ve<\gamma_0$, where $\gamma_0$ is given in Theorem \ref{stability-general}.
Let $\beta=\cmr-\pi_\lambda(F)$. Let $\gamma=\min\{\ve, \frac{\beta}{3r}\}$.
Let $\delta, n_0$ be the constants guaranteed by Theorem \ref{stability-general} for the above defined $\gamma$.
Let $\delta_1=\min\{\delta, \frac{\beta}{3}\}$. Let $n_1\geq n_0$ to be large enough so that for $n\geq n_1$ we have
$$(\cmr-\delta_1-\gamma r)\binom{n}{r}>(\cmr-\frac{2\beta}{3}) \binom{n}{r}>
\pi_\lambda(F) \frac{n^r}{r!}.$$
Let $G$ be an $\cH_{m+1}^F$-free graph of order $n\geq n_1$ and size
more than $(\frac{[m]_r}{m^r}-\delta_1) \binom{n}{r}$.
Let $G^*$ the final graph produced by applying Algorithm \ref{symm-clean} to $G$ with threshold $\gamma$. By Theorem \ref{stability-general}, $n(G^*)\geq (1-\gamma )n\geq (1-\ve)n$. 
Since $G^*$ is the final graph produced by Algorithm \ref{symm-clean},
if $S$ consists of one vertex from each equivalence class of $G^*$ then $G^*[S]$ covers pairs and $G^*$ is a blowup of $G^*[S]$.
If $|S|\leq m$, then $W=V(G^*)$ is the union of at most $m$ equivalence classes
of $G^*$. By Theorem \ref{stability-general}, $G[W]$ is $m$-partite. So $G$ can be made
$m$-partite by deleting at most $\ve n$ vertices and we are done.

We henceforth assume that $|S|\geq m+1$. If $F\subseteq G^*[S]$, then since
$G^*[S]$ covers pairs we can find a member of $\cH_{m+1}^F$ in $G^*[S]$
by using any $(m+1)$-set that contains a copy of $F$ as the core, contradicting $G^*$ being $\cH_{m+1}^F$-free. Hence $G^*[S]$ is $F$-free. In the producing $G^*$ from $G$, observe that each time we symmetrized, the number of edges does not decrease. 
Since at most $\gamma n$ vertices are deleted in the process, 
$$|G^*|>|G|-\gamma n \binom{n-1}{r-1}\geq (\cmr-\delta_1-\gamma r)\binom{n}{r}>\pi_\lambda(F) \frac{n^r}{r!},$$
contradicting Lemma \ref{blowup}.  \qed


\section{Establishing Exactness from Stability}

In this section, we prove Theorem \ref{stable-to-exact}: Let $F$ be an $r$-graph such that either $n(F)\leq m$ or $n(F)=m+1$ and $F$
contains a vertex of degree $1$. We prove that if $\cH_{m+1}^F$ is $m$-stable then  $ex(n,H^F_{m+1})=|\Tr (n,m)|$
for sufficiently large $n$.

\medskip

{\bf Proof of Theorem \ref{stable-to-exact}:}
First we define a few constants.
Let 
\begin{equation} \label{constant3}
c_1=\frac{1}{(2m)^{m^3}}, \quad c_2=\frac{c_1}{2mr}, \quad c_3=\frac{c_1}{r}, \quad c_4=\frac{c_1}{2r\binom{m-1}{r-2}}, \quad c_5=(\frac{c_4}{2m})^{m^3}.
\end{equation}
Let  
\begin{equation} \label{epsilon-definition}
\ve=\min\{\frac{c_5}{2r}, \frac{c_1c_2}{2r^2}\}.
\end{equation}
Since $\cH_{m+1}^F$ is $m$-stable, by Definition \ref{stable},  there exist a real $\delta_1>0$ and a positive integer $n_1$ such that for all $n\geq n_1$ if $G$ is an $\cH_{m+1}^F$-free $r$-graph on $[n]$
with $|G|>(\cmr-\delta_1)\binom{n}{r}$ edges then $G$ can be made $m$-partite
by deleting at most $\ve n$ vertices. By further reducing $\delta_1$ if needed, we
may assume that $\delta_1\leq \ve$.
 Let $n_2$ be sufficiently large so that $n_2\geq n_1$ and that every $n\geq n_2$
satisfies various inequalities involving $n$ that we will specify throughout the proof.
Let $G$ now be a maximum $H^F_{m+1}$-free graph on $[n]$ with $n\geq n_2$.  Since $\Tr(n,m)$ is $H_{m+1}^F$-free, we have
\begin{equation}  \label{maximum}
|G|\geq |\Tr(n,m)|.
\end{equation}
Let $p=n(\cH_{m+1}^F)$.
By Lemma \ref{removal}, $G$ contains a subgraph $G'$ with
$|G'|\geq |G|-p\binom{n}{r-3}\binom{n}{2}$ such that $G'$ is
$\cH_{m+1}^F$-free.
For large enough $n\geq n_2$ we have $|G'|>(\cmr-\delta_1)\binom{n}{r}$.
Since $\cH^F_{m+1}$ is $m$-stable and $n(G')\geq n_1$,
$G'$ can be made $m$-partite by deleting at most $\ve n$ vertices. Hence, in particular, $G'$ contains an $m$-partite subgraph with at least 
$$|G'|-\ve n^r\geq |G|-p\binom{n}{r-3}\binom{n}{2} -\ve n^r\geq |G|-2\ve n^r$$ edges (assuming that $n_2$ is large enough).
Among all $m$-partitions of $[n]$, let $V_1\cup \ldots \cup V_m$ be an $m$-partition of $[n]$ that maximizes

\begin{equation} \label{phi-definition}
\phi=\sum_{e\in G}|\{i\in [m]: e\cap V_i\neq \emptyset\}|.
\end{equation}

Let $K$ be the complete $m$-partite $r$-graph on $[n]$ with parts $V_1, \ldots, V_m$. 
By the definition of $\phi$, we have $\phi\geq r|G\cap K|$.
By the choice of $K$, we have $|G\cap K|\geq |G|-2\ve n^r$ and thus $\phi\geq r(|G|-2\ve n^r)$. On the other hand, $\phi\leq r|G|-|G\setminus K|$,
since each edge of $G\setminus K$ contributes at most $r-1$ to $\phi$.
It follows that
\begin{equation} \label{bad-bound}
|G\setminus K|\leq 2r\ve n^r.
\end{equation}

We call an edge $e$ on $[n]$ {\it crossing} if it contains at most vertex of each $V_i$, i.e. if $e\in K$. Let 
$$M=K\setminus G  \quad \mbox{ and } \quad B=G\setminus K.$$
We call edges in $M$ {\it missing edges}. We call edges in $B$ {\it bad edges}. 
Since $|G|\geq |\Tr(n,m)|\geq |K|$,
we have $|B|\geq |M|$. By \eqref{bad-bound}, we then have
\begin{equation} \label{M-B}
|M|\leq |B|\leq 2r\ve n^r.
\end{equation}

Our goal for the rest of the proof is to show that in fact $B=\emptyset$, from which we would have $|G|\leq |K|\leq |\Tr(n,m)|$, which would complete  our proof.
For the rest of the proof,  we suppose $B\neq \emptyset$ and will derive a contradiction.

First, note that 
\begin{equation}
|K|\geq |K\cap G|=|G|-|G\setminus K|\geq |\Tr(n,m)|-2r\ve n^r,
\end{equation}
for sufficiently large $n$.
For sufficiently large $n$, this implies 

\begin{equation} \label{part-size}
\forall i \in [m], \quad 0.9\frac{n}{m}\leq |V_i|\leq 1.1\frac{n}{m}.
\end{equation}

Let $q=\binom{m+1}{2}+r$. Let $K^r_m(q)$ denote the complete $m$-partite
$r$-graph with $q$ vertices in each part. Let $A_1,\ldots, A_m$ denote the $m$ parts.  

\medskip

{\bf Claim 3.} If $u,v$ are two vertices in some part of a copy $L$ of $K^r_m(q)$ in $G$. Then $d_G(\{u,v\})=0$.

\medskip

{\it Proof of Claim 3.}  Without loss of generality, suppose $u,v$ lie in $A_1$. 
Suppose for contradiction that $u,v$ lie in some edge $e$ of $G$.
Let $C$ denote the core of $H^F_{m+1}$. By our assumption about $F$,
there exists $z\in C$ such that $z$ lies in $0$ or $1$ edge of $F$. Since $m+1\geq r$,
there exists $y\in C\setminus \{z\}$ such that $d_F(\{y,z\})=0$.
We can obtain a copy $H^F_{m+1}$ in $G$ by mapping $y,z$ to $u,v$, respectively, and the other vertices of $C$ into $A_2,\ldots, A_m$, one into each part. It remains to cover the pairs in $C$ that are uncovered by $F$. The pair $\{x,y\}$ is covered by $e$.
Since each part of $L$ still has at least $\binom{m+1}{2}$ vertices outside $e$, it is easy to cover
all such pairs so that the covering edges are pairwise disjoint outside $C$.
This contradicts $G$ being $H^F_{m+1}$-free. \qed

\medskip

{\bf Claim 4.} Let $e\in B$ and suppose $|e\cap V_i|\geq 2$. Let $u,v\in e\cap V_i$. Then either $d_M(u)\geq c_1 n^{r-1}$ or
$d_M(v)\geq c_1 n^{r-1}$.

\medskip

{\it Proof of Claim 4.} Without loss of generality, suppose $i=1$. 
Let $S$ be a set of $mq$ vertices obtained by selecting $u, v$, $q-2$ vertices from $V_1\setminus e$ and $q$ vertices from each of $V_2\setminus e,\ldots, V_m\setminus e$. By Claim 3, $K[S]\not\subseteq G[S]$.
Hence, for each such $S$, $\exists f\in K[S]\setminus G[S]\subseteq M$. There are at least
$(n/2m)^{mq-2}$ choices of $S$. Suppose first that for at least half of the choices of $S$, the corresponding $f$  is disjoint from $\{u,v\}$. Each fixed $f\in M$, 
there are at most $n^{mq-2-r}$ choices of $S$ 
for which $f\in K[S]\setminus G[S]$. Hence, using the definition of $c_1$ in \eqref{constant3} and the definition of $\ve$ in \eqref{epsilon-definition}, we have
$$|M|\geq (1/2)(n/2m)^{mq-2}/n^{mq-2-r}\geq c_1 n^r>2r\ve n^r,$$
 contradicting (\ref{M-B}). 
Next, suppose  for at least half of the choices of $S$,
$K[S]\setminus G[S]$, $f$ intersects $\{u,v\}$.
Without loss of generality, suppose for at least $1/4$ of the choices of $S$, $f$ contains $u$. Since each such $f$ is contained in at most $n^{mq-1-r}$ many $S$, there are at least
$$(1/4)(n/2m)^{mq-2}/n^{mq-1-r}\geq c_1n^{r-1},$$
edges of $M$ that contains $u$.
\qed

\medskip
Let
$$W=\{w: d_M(w)\geq c_1n^{r-1}\}.$$
By the definition of $W$ and \eqref{M-B}, we have
$$c_1n^{r-1}|W|\leq \sum_{x\in [n]} d_M(x)=r|M|\leq 2r^2\ve n^r.$$
Hence
\begin{equation} \label{W-upper}
|W|\leq \frac{2r^2\ve}{c_1} n <c_2 n,
\end{equation}
where the last inequality follows from \eqref{constant3} and \eqref{epsilon-definition}.
We call a  pair $\{u,v\}$ of vertices  in $G$  a {\it bad pair} if  $\exists \, e\in B, i\in [m]$ such that $u,v\in e\cap V_i$. By Claim 4, each bad pair must contain an element of $W$. Hence, by \eqref{W-upper}, we have

\medskip

{\bf Claim 5.} The number of bad pairs in $G$ is at most $c_2 n^2$.

\medskip

We call vertices in $W$ {\it defect vertices}.
By Claim 4, each $e\in B$ contains a defect vertex in a part $V_i$ where $|e\cap V_i|\geq 2$. We pick such a defect vertex, denote it by $c(e)$ and call it the 
{\it center} of $e$. (The choices for $c(e)$ may not be
unique, but we will fix one.)
For each defect vertex $w$, let $b(w)$ denote the number of edges $e\in B$ such
that $c(e)=w$. By our discussion above and \eqref{M-B}, we have

$$\sum_{w\in W} b(w)\geq |B|\geq |M|=\frac{1}{r}\sum_{x\in [n]} d_M(x)\geq \frac{1}{r}\sum_{w\in W} d_M(w).$$

Hence there exists $w_0\in W$ such that
\begin{equation} \label{bad-vertex}
b(w_0)\geq \frac{1}{r} d_M(w_0)\geq \frac{c_1}{r} n^{r-1}=c_3 n^{r-1}.
\end{equation}
Without loss of generality, may assume that $w_0\in V_1$. 
Let 
$$L=\{e\setminus w_0: e\in B, c(e)=w_0\}.$$

By definition, any $e\in B$ with $c(e)=w_0$ contains
at least one other vertex of $V_1$. Hence
\begin{equation} \label{v1-intersection}
\forall f\in L, f\cap V_1\neq \emptyset.
\end{equation}

By Claim 5, for each $i\in [m]$, the number of members of $L$ that contain two or more vertices of $V_i$ is  at most $c_2n^2\cdot n^{r-3}=c_2n^{r-1}$. Hence, using the defintion of the constants given in \eqref{constant3}, the number of members of $L$ that contains at most
one vertex from each $V_i$ is at least 
$$c_3n^{r-1}-mc_2n^{r-1}=(\frac{c_1}{r}-m\frac{c_1}{2mr})n^{r-1}=\frac{c_1}{2r}n^{r-1}=c_4\binom{m-1}{r-2}n^{r-1}.$$

Each such member of $L$ contains a vertex of $V_1$ by \eqref{v1-intersection}.
By the pigeonhole principle, for some collection of $r-2$ parts
outside $V_1$, without loss of generality, say $V_2,\ldots, V_{r-1}$, there are at least $c_4n^{r-1}$ members
of $L$ that contain exactly one vertex of each of $V_1, V_2,\ldots, V_{r-1}$.
Let $L'$ denote the collection of these members. Then $L'$ is an  $(r-1)$-partite $(r-1)$-graph with an $(r-1)$-partition $U_1,\ldots, U_{r-1}$, where $\forall i\in [r-1],
U_i=V(L')\cap V_i$, and 
\begin{equation} \label{L-prime-lower}
|L'|\geq c_4n^{r-1}.
\end{equation}

Recall that $p=n(H_{m+1}^F)$. Assuming $n$ is sufficiently large,
by Lemma \ref{kernel-degree}, $L'$ has a subgraph $L''$ with 
\begin{equation} \label{L-prime-prime-lower}
|L''|\geq |L'|-p n\binom{n}{r-3}\geq \frac{1}{2}c_4n^{r-1}
\end{equation}
such that for each vertex $x$ with $d_{L''}(x)>0$ we have  $d^*_{L''}(x)>p$.
Let us remove isolated vertices from $L''$. Then the condition implies
that for 

\medskip
\begin{equation} \label{L''-kernel}
\forall x\in V(L''), d^*_{G}(\{w_0,x\})\geq p+1.
\end{equation}

For each $i=r,\ldots, m$, let $$D_i=\left\{x\in V_i: d_G(\{x,w_0\})>p{n-3\choose r-3}\right\}.$$

\medskip

By the definition of $D_i$ and Lemma \ref{matching}, we have
\begin{equation} \label{Di-kernel}
\forall i=r+1,\ldots, m,  \forall x\in D_i, d^*_{G}(\{w_0,x\})\geq p+1.
\end{equation}

{\bf Claim 8.} For each $i=r,\ldots, m$, $|D_i|\geq \frac{1}{2}c_4n$.

\medskip

{\it Proof of Claim 8.} Suppose for contradiction that for some $i$, $|D_i|<\frac{1}{2}c_4n$. Without loss of generality,
suppose that $|D_m|<\frac{1}{2}c_4n$. Let us consider a different $m$-partition of $V(G)$ by
moving $w_0$ from $V_1$ to $V_m$. Let us consider the change to the value of $\phi$,
defined in (\ref{phi-definition}). The only edges $e$ of $G$ whose contribution to $\phi$ are decreased by the move are those satisfying $e\cap V_1=\{w_0\}$ and $e\cap V_m
\neq \emptyset$. The decrease is $1$ per edge. 
We can bound the number of such edges as follows.
The number of edges of $G$ containing $w_0$ and a vertex of $D_m$ is at most $|D_m|\cdot n^{r-2}<\frac{1}{2}c_4n^{r-1}$.
For each vertex $x$ of $V_m\setminus D_m$ we have
$d_G(\{x,w_0\})<p{n-1\choose r-3}$). Hence the number of edges of $G$ containing
$w_0$ and a vertex of $V_m\setminus D_m$ is at most
 $n\cdot p{n-1\choose r-3}<\frac{1}{4}c_4n^{r-1}$, for sufficiently large $n$. Hence there are fewer than $\frac{1}{2}c_4n^{r-1}+\frac{1}{4}c_4n^{r-1}<c_4n^{r-1}$ such edges.

On the other hand, the contribution to $\phi$ from each in $\{w_0\cup f: f\in L'\}$
is increased by $1$ by moving $w_0$ to $V_m$. By \eqref{L-prime-lower}, 
$|L'|\geq c_4n^{r-1}$.
Hence $V_1\setminus, V_2,\ldots, V_{m-1}, V_{m+1}\cup w_0$ is a partition that has a higher $\phi$-value than $V_1,\ldots, V_m$, contradicting our choice of $V_1,\ldots, V_m$. 
\qed

\medskip

We are now ready to complete the proof of the theorem. 
Let $\cS$ be the collection of all $mq$-sets $S$ obtained by picking the vertex set of an edge $f$ of $L''$, and then picking $q-1$ vertices from $V_i\setminus f$, for each $i\in [r-1]$, then picking $q$ vertices from each of $D_r,\ldots, D_m$ 
By \eqref{constant3}, \eqref{part-size}, \eqref{L-prime-prime-lower}, and Claim 8,
$$|\cS|\geq  \frac{1}{2}c_4n^{r-1}\cdot \left (\frac{n}{2m}\right)^{(r-1)(q-1)}\cdot
\left(\frac{1}{2}c_4n\right)^{(m-r+1)q}\geq c_5n^{mq}.$$

\medskip

{\bf Claim 9.}  For each $S\in \cS$, $K[S]\not\subseteq G[S]$.

\medskip

{\it Proof of Claim 9.}  Suppose for contradiction that $K[S]\subseteq G[S]$. 
Let $C$ denote the core of $H^F_{m+1}$. By our assumption, $C$ contains
a vertex $z$ that lies in $0$ or $1$ edge of $F$.
Let $f$ be a member of $L''$ contained in $S$. Let $S'\subseteq S$ contain
$V(f)$ and one vertex from each of $V_r,\ldots, V_m$. By our assumption,
$G[S']$ is complete. Thus $G[S'\cup w_0]$ contains a copy $F'$ of $F$, where $w_0$ plays the role of $z$ and  if $z$ has degree $1$ in $F$ then $w_0\cup f$ plays the role of the unique edge of $F$ containing $z$. With $C=S'\cup \{w_0\}$, we can obtain a copy of $H_{m+1}^F$ in $G$ as follows. It suffices to cover the
pairs $\{a,b\}$ in $C$ that are uncovered by $F'$ using edges that intersect $C$ only in
$a,b$ and are pairwise disjoint outside $C$. Let $\{a,b\}$ be such pair. If $a,b\neq w_0$,
then we can use an edge in $G[S]\supseteq K[S]$ to cover $\{a,b\}$ like in the proof of Claim 3. To cover a pair of the form $\{w_0,a\}$, we use  \eqref{L''-kernel} if $a\in V(f)$ or \eqref{Di-kernel} if $a\in S'\setminus V(f)$. Hence $H_{m+1}^F\subseteq G$, a contradiction.
\qed

\medskip

Now, for each $S\in \cS$, by Claim 9, $K[S]$ contains a member of $K\setminus G$. On the other hand,  each member of $K\setminus G$ trivially is contained in at most $n^{mq-r}$ different $S$. Hence, $$|K\setminus G|\geq |\cS|/n^{mq-r}\geq c_5n^{mq}/n^{mq-r} =c_5n^r\geq 2r\ve n^r,$$ contradicting (\ref{M-B}). The contradiction completes our proof.
\qed


\section{Stability of expanded cliques with embedded enlarged trees}

\medskip

In this section, we prove Theorem \ref{enlarged-tree-stable}. The main work in proving
Theorem \ref{enlarged-tree-stable} is to establish stability of the Lagrangians of
trees, which may be of independent interest.
Given an $r$-graph $G$ on $[t]$ and variables $\tx=(x_1,\ldots, x_t)$, recall that $$p_G(\tx)=r!\cdot\sum_{e\in G} \prod_{i\in e} x_i,$$
and that $\lambda(G)$ is the maximum $p_G(\tx)$ over all $1$-sum weight assignments $\tx$.
For each $i\in [t]$, let $\lambda_i=\frac{\partial (p_G(\tx))}{\partial (x_i)}$. Then it is straightforward to verify that

\begin{equation} \label{weighted-sum}
\lambda_i=r!\cdot \sum_{f\in \cL_G(i)}\prod_{j\in f} x_j\quad\quad\mbox{and}\quad \quad
p_G(\tx)=\frac{1}{r}\sum_{i=1}^n  \lambda_i x_i.
\end{equation}
By \eqref{weighted-sum}, we have
\begin{equation} \label{pg-lambda}
p_G(\tx)\leq \frac{1}{r}\max_i \lambda_i \quad \mbox{ and thus} \quad
\max_i \lambda_i \geq r\cdot p_G(\tx).
\end{equation}

The following lemma will be useful for our analysis.

\begin{lemma} \label{leader}
Let $\eta>0$ be a real.
Let $G$ be an $r$-graph on $[t]$ 
and $\tx=(x_1,\ldots, x_t)$ a $1$-sum weight assignment on $G$
with $p_G(x)\geq \lambda(G)-\eta$. Then there exists
$i\in [t]$ such that $x_i\geq \max_{j\in [t]} x_j -2r!\sqrt{\eta}$ and $\lambda_i\geq \max_{j\in[t]} \lambda_j-2r!\sqrt{\eta}$.
\end{lemma} 
\begin{proof}
Suppose $x_a=\max_j x_j$ and $\lambda_b=\max_j \lambda_j$.
If $a=b$ then the claim holds  with $i=a=b$. So assume $a\neq b$.
If $\lambda_a>\lambda_b-2r!\sqrt{\eta}$, then the claim holds with $i=a$.
Similarly, if $x_b>x_a-2r!\sqrt{\eta}$, then the claim holds with $i=b$.
So we may assume that  $\lambda_a<\lambda_b-2r!\sqrt{\eta}$ and
$x_b<x_a-r!\sqrt{\eta}$. That is, $x_a-x_b>2r!\sqrt{\eta}$ and
$\lambda_b-\lambda_a>2r!\sqrt{\eta}$.

 Let $w_a=\sum_{e\in \cL(a)\setminus \cL(b)}\prod_{i\in e} x_i$,
 $w_b=\sum_{e\in \cL(b)\setminus \cL(a)}\prod_{i\in e} x_i$, and $w^*=\sum_{e\in \cL(\{a,b\})}\prod_{i\in e} x_i$. It is easy to see that $0\leq w_a,w_b,w^*\leq 1$.
Note that $\lambda_a=r!(w_a+x_bw^*)$ and $\lambda_b=r!(w_b+x_aw^*)$. Hence,
$\lambda_b-\lambda_a=r![w_b-w_a+(x_a-x_b)w^*]$.
Let $d=\frac{1}{2r!}(\lambda_b-\lambda_a)$.
Consider a new weight assignment $\ty=(y_1,\ldots, y_n)$ defined by letting
$y_a=x_a-d, y_b=x_b+d$ and $\forall i \in [t]\setminus \{a,b\}\, 
y_i=x_i$. Then
\begin{eqnarray*}
p_G(\ty)-p_G(\tx)&=&r![(x_a-d)(x_b+d)-x_ax_b)w^*-dw_a+dw_b].\\
&=&r![d((x_a-x_b)w^*+(w_b-w_a)]-d^2w^*]\\
&\geq& r!d[((x_a-x_b)w^*+(w_b-w_a)]-d]\\
&=&d[\lambda_b-\lambda_a-r!d]=\frac{1}{2}d(\lambda_b-\lambda_a)
=\frac{1}{4r!}(\lambda_b-\lambda_a)^2>\eta,
\end{eqnarray*}
contradicting our assumption that $p_G(\tx)\geq \lambda(G)-\eta$.
\end{proof}

Given $0<\beta\leq 1$ and an $r$-graph $G$ on $[t]$, let 
$$\lambda_\beta(G)=\max \{p_G(x_1,\ldots, x_t):  \forall i\in [t] \, x_i\geq 0, \sum_i x_i=1,
\max_i x_i=\beta\}.$$
Clearly, $$\lambda(G)=\max_{0<\beta\leq 1} \lambda_\beta(G).$$
The following lemma  played a crucial role in Sidorenko's arguements.

\begin{lemma}{\rm (\cite{sidorenko} Lemma 3.3)} \label{local}
Let $k,r\geq 2$ be integers, where $k\geq M_{r-1}$. Let $0<\beta\leq 1$ be a real.
Let $T$ be a $k$-vertex tree that satisfies the Erd\H{o}s-S\'os conjecture.
Let $T$ be the $(r-2)$-fold enlargement of $T$.
If $G$ is an $F$-free $r$-graph  then $\lambda_\beta(G)\leq (k-2)f_r(z)$, where $z=\max\{\frac{1}{\beta}-r+3, k\}$.
\end{lemma}

Let us also mention a useful fact about the function $f_r(x)$, namely
\begin{equation} \label{f-fact}
\frac{f_r(x)}{f_{r-1}(x)}=\left(\frac{x+r-4}{x+r-3}\right)^{r-1}.
\end{equation}

Also, let us recall the well-known fact that $\pi_\lambda(K_k)=\pi(K_k)=\frac{k-2}{k-1}$
(see \cite{ES} and \cite{MS} for instance).
\begin{lemma}\label{representing-vertex}
Let $k\geq 3$.
Let $T$ be a $k$-vertex tree that satisfies the Erd\H{o}s-S\'os conjecture and $F$ the $(r-2)$-enlargement of $T$, where $r\geq 2$ and $k\geq M_r$.
For every real $\alpha>0$, there exists a real $\gamma=\gamma(\alpha)>0$ such that
if $G$ is an $F$-free $r$-graph on $[t]$ and $\tx=(x_1,\ldots, x_t)$ a $1$-sum
weight assignment on $G$ with $p_G(\tx)\geq (k-2)f_r(k)-\gamma$ then
there exists $i\in [t]$ such that 
\begin{enumerate}
\item $|x_i-\frac{1}{k+r-3}|<\a$.
\item $|\sum_{j\in V(\cL_G(i))} x_j-\frac{k+r-4}{k+r-3}|<\a$.
\item $\lambda_i>r(k-2) f_r(k)-\alpha$.
\end{enumerate}
\end{lemma}
\begin{proof}
Let $\a>0$ be given. Choose a small enough real $d>0$ such that
$\frac{1}{k+r-3+d}>\frac{1}{k+r-3}-\frac{\a}{2}$.
Since $k>M_r$ and $M_r$ is the right most local max
of $f_r(x)$, $f_r(x)$ is strictly decreasing on $[k,\infty)$.
Choose $\gamma>0$ to be small enough so that
\begin{equation} \label{gamma-choices}
\gamma<f_r(k)-f_r(k+d), \quad 3r!\sqrt{\gamma}<\frac{\a}{2},\quad
\mbox{ and }  (1-\frac{3r!\sqrt{\gamma}}{r(k-2)f_r(k)})^{\frac{1}{r-1}}>1-\a.
\end{equation}

Let $\beta=\max_j x_j$ and $\lambda_{max}=\max_j \lambda_j$.
If $\frac{1}{\beta}-r+3>k+d$
and by Lemma \ref{local}, 
$$p_G(\tx)\leq \lambda_\beta(G)\leq (k-2) f_r(k+d)<(k-2)[f_r(k)-\gamma]<(k-2)f_r(k)-\gamma,$$ contradicting our assumption. Hence
$\frac{1}{\beta}-r+3\leq k+d$. Solving for $\beta$ we have, by our choice of $d$, that
$$\beta\geq \frac{1}{k+r-3+d}>\frac{1}{k+r-3}-\frac{\a}{2}.$$
By lemma \ref{leader}, for some $i\in [n]$, say $i=1$, we have
\begin{equation} \label{x1lambda1}
x_1\geq \beta-2r!\sqrt{\gamma} \mbox{ and } \lambda_1\geq \lambda_{max}-2r!\sqrt{\gamma}.
\end{equation}
Since $3r!\sqrt{\gamma}<\frac{\alpha}{2}$ by our choice of $\gamma$, 
\begin{equation} \label{x1-lower}
x_1\geq \frac{1}{k+r-3}-\frac{\a}{2}-2r!\sqrt{\gamma}\geq \frac{1}{k+r-3}-\a.
\end{equation}
By \eqref{pg-lambda}, \eqref{x1lambda1}, and our assumption that $p_G(\tx)\geq (k-2)f_r(k)-\gamma$,
\begin{equation} \label{lambda1-lower}
\lambda_1\geq \lambda_{max}-2r!\sqrt{\gamma}\geq r((k-2)f_r(k)-\gamma)-2r!\sqrt{\gamma}>r(k-2)f_r(k)-3r!\sqrt{\gamma}\geq r(k-2)f_r(k)-\a.
\end{equation}
This proves item 3. Next, we prove that $\sum_{j\in V(\cL_G(i))}x_i>\frac{k+r-4}{k+r-3}-\a$.
If $r=2$ then $\lambda_1=2\sum_{j\in N_G(1)} x_j$ and hence
$\sum_{j\in V(\cL_G(1))} x_j=\frac{\lambda_1}{2}>(k-2)f_2(k)-\frac{\a}{2}\geq\frac{k-2}{k-1}-\a$.
We henceforth assume that $r\geq 3$.
Let $s=\sum_{j\in \cL_G(1)} x_j$. For each $j\in V(\cL_G(1))$ let $y_j=\frac{x_j}{s}$.
Then $\sum_{j\in V(\cL_G(1)} y_j=1$. Let $\ty$ denote the $1$-sum weight assignment on $\cL_G(1)$ defined by the $y_j$'s. Let $F'$ denote
the $(r-3)$-enlargement of $T$. Since $G$ is $F$-free, $\cL_G(1)$ is $F'$-free.
Since $F'$ is the $(r-3)$-enlargement of $T$, where $T$ is a $k$-vertex tree
satisfying the Erd\H{o}s-S\'os conjecture and $k\geq M_r\geq M_{r-1}$,
by Theorem \ref{sid-main}, $\pi_{\lambda}(F')\leq (k-2)f_{r-1}(k)$.
Hence,
\begin{equation} \label{lambda1-upper}
\lambda_1=r!\sum_{e\in\cL_G(1)}\prod_{j\in e}x_j=rs^{r-1}\cdot (r-1)!\sum_{e\in \cL_G(1)}\prod_{j\in e} y_j\leq rs^{r-1}\cdot (k-2)f_{r-1}(k).
\end{equation}
By \eqref{lambda1-lower} and \eqref{lambda1-upper}, we have
$$rs^{r-1}\cdot (k-2)f_{r-1}(k)> 
r(k-2)f_r(k)(1-\frac{3r!\sqrt{\gamma}}{r(k-2)f_r(k)}).$$
Using \eqref{f-fact}, we have
$$s^{r-1}>\left(\frac{k+r-4}{k+r-3}\right)^{r-1}(1-\frac{3r!\sqrt{\gamma}}{r(k-2)f_r(k)}).$$
Hence by our choice of $\gamma$ given in \eqref{gamma-choices}, we have 
\begin{equation} \label{xj-sum-lower}
\sum_{j\in V(\cL_G(1))} x_j=s>\frac{k+r-4}{k+r-3}\left(1-\frac{3r!\sqrt{\gamma}}{r(k-2)f_r(k)}\right)^{\frac{1}{r-1}}
>\frac{k+r-4}{k+r-3}(1-\a)>\frac{k+r-4}{k+r-3}-\a.
\end{equation}
Now, \eqref{x1-lower} and \eqref{xj-sum-lower} together prove item 2 and item 3.
\end{proof}

We need another lemma from \cite{sidorenko}. Given a graph $G$, let
$d(G)=\max_{H\subseteq G} \frac{2e(H)}{n(H)}$. So, $d(G)$ is the maximum
average degree of a subgraph of $G$ over all subgraphs of $G$.

\begin{lemma} \label{max-1} {\rm (\cite{sidorenko} Theorem 2.4)}
Let $G$ be a graph on $[t]$. Let $\ty=(y_1,\ldots, y_t)$ be a weight assignment on $G$ where $\max_i y_i=1$. Then $\frac{p_G(\ty)}{\sum_{i=1}^t y_i}\leq d(G)$.
\end{lemma}
\begin{corollary} \label{local-max}
Let $G$ be a graph on $[t]$. Let $\tx=(x_1,\ldots, x_t)$ be a $1$-sum weight assignment on $G$ where $\max_i x_i=\beta$. Then $p_G(\tx)\leq \beta d(G)$.
\end{corollary}
\begin{proof} For each $i\in [t]$, let $y_i=x_i/\beta$. Then $\max_i y_i=1$ and
$\sum_i y_i=1/\beta$. Using Lemma \ref{max-1}, we have $p_G(\tx)=\beta^2 p_G(\ty)
\leq \beta^2 d(G)\sum_i y_i=\beta d(G)$.
\end{proof}

\begin{lemma} \label{join-degree}
Let $H$ be a graph with average degree $d$ and $G$ obtained from $H$ by adding a new vertex and making it adjacent to all of $V(H)$. Then $G$ has average degree
at least $d+1$.
\end{lemma}
\begin{proof} Suppose $H$ has $p$ vertices. Clearly $p\geq d+1$. We have $n(G)=p+1$ and $e(G)=pd/2+p$.
So $G$ has average degree $\frac{2e(G)}{n(G)}=\frac{pd+2p}{p+1}\geq \frac{pd+p+d+1}{p+1}=d+1$.
\end{proof}

\begin{lemma} \label{average-degree-stability}
Let $d$ be a positive integer. 
Let $0<\ve<1$ be a real. There exists a real $\delta_d(\ve)>0$ such that if $G$ a graph on $[t]$ with $d(G)\leq d$ and $\tx=(x_1,\ldots x_t)$ is a $1$-sum weight assignment on $G$ with $p_G(\tx)\geq \frac{d}{d+1}-\delta_d$
then $\exists I\subseteq [t]$ with $|I|\leq d+1$ such that $\sum_{i\in I} x_i\geq 1-\ve$.
\end{lemma}
\begin{proof}
We use induction on $d$. For the basis step, let $d=1$. 
Let $\delta_1(\ve)=\frac{\ve}{2}$. Suppose $d(G)\leq 1$ and $p_G(\tx)\geq \frac{1}{2}-\delta_1$.
By our assumption about $G$, each nontrivial component of $G$ is a single edge.  Suppose  that there are $s$ nontrivial components with total vertex weights $w_1,\ldots, w_s$, respectively, where without loss of generality, suppose $w_1=\max_j w_j$.
We have

$$\frac{1}{2}-\delta_1(\ve) \leq p_G(x)\leq 2!\sum_{i=1}^s \frac{w_i^2}{4}\leq \frac{1}{2}w_1(w_1+w_2+\ldots+w_s)\leq \frac{1}{2}w_1.$$ 

Hence $w_1\geq 1-2\delta_1(\ve)=1-\ve$, implying that there exists $I\subseteq [t], |I|=2$ with $\sum_{i\in I} x_i\geq 1-\ve$.

For the induction step, let $d\geq 2$. Choose a small real $\a$ such that
$$0<\a<\frac{\ve}{4}$$
and
\begin{equation}\label{alpha-definition1}
\left(\frac{d}{d+1}+\a\right)^{-2}\cdot\left(\frac{d(d-1)}{(d+1)^2}-7\a \right)>\frac{d-1}{d}-\delta_{d-1}(\frac{\ve}{2}).
\end{equation}
Choose $0<\delta_d<\a-2\a^2$ to be be sufficiently small such that $$\delta_d+2r!\sqrt{\delta_d}<\a.$$ 

Suppose $p_G(\tx)\geq \frac{d}{d+1}-\delta_d$. 
By Lemma \ref{leader}, there exists $i\in [t]$, say $i=1$, such that
$x_1\geq \max_i x_i-2r!\sqrt{\delta_d}$ and $\lambda_1\geq \max_i \lambda_i -2r!\sqrt{\delta_d}$. By \eqref{pg-lambda},
we have $\max_i \lambda_i\geq 2p_G(\tx)\geq \frac{2d}{d+1}-2\delta_d$
and hence $\lambda_1\geq \frac{2d}{d+1}-2\delta_d-2r!\sqrt{\delta_d}$. 
Since $\lambda_1=2\sum_{j\in N_G(1)}x_j$, this also yields
\begin{equation} \label{neighborhood-lower}
\sum_{j\in N_G(1)}x_j\geq \frac{d}{d+1}-\delta_d-r!\sqrt{\delta_d}\geq
\frac{d}{d+1}-\a.
\end{equation}
Hence $x_1\leq \frac{1}{d+1}+\a$.
Let $\beta=\max_i x_i$.  By Corolloary \ref{local-max}, we have
$\frac{d}{d+1}-\delta_d\leq p_G(\tx) \leq \beta d$.
Hence, $\beta\geq \frac{1}{d+1}-\frac{\delta_d}{d+1}$ and thus
\begin{equation} \label{x1-lower2}
x_1\geq \beta-2r!\sqrt{\delta_d}\geq \frac{1}{d+1}-\frac{\delta_d}{d+1}-2r!\sqrt{\delta_d}\geq
\frac{1}{d+1}-\a.
\end{equation}
Let $N=N_G(1)$ and $\overline{N}=[t]\setminus (N\cup \{1\})$.
By \eqref{neighborhood-lower} and \eqref{x1-lower2},
\begin{equation} \label{non-N}
\sum_{j\in \overline{N}} x_j<2\a.
\end{equation}
Since $1$ is adajcent to all of $N$ and $d(G)\leq d$, by Lemma \ref{join-degree},
$d(G[N])\leq d-1$. Let $s=\sum_{j\in N} x_j$. By \eqref{neighborhood-lower} and 
\eqref{x1-lower2}, we have $\frac{d}{d+1}-\a\leq
s\leq \frac{d}{d+1}+\a$. Since $\lambda_1=2s$, we also have
$2(\frac{d}{d+1}-\a)\leq \lambda_1\leq 2(\frac{d}{d+1}+\a)$.
For each $j\in N$ let $y_j=\frac{x_j}{s}$.
Then $\sum_{j\in N} y_j=1$. Let $\ty$ denote the $1$-sum weight assignment on $G[N]$ given by the $y_j$'s. We have, using the upper bounds on $\lambda_1,x_1$, and
\eqref{non-N},
\begin{eqnarray*}
s^2p_{G[N]}(\ty)=2\sum_{\{i,j\}\in G[N]}  x_ix_j&\geq& p_G(\tx)-\lambda_1 x_1-2 \sum_{j\in [t]} x_j \sum_{j\in \overline{N}} x_j\\
&\geq&(\frac{d}{d+1}-\delta_d)-2(\frac{d}{d+1}+\a)(\frac{1}{d+1}+\a)-4\a\\
&=&\frac{d(d-1)}{(d+1)^2}-\delta_d-6\a-2\a^2\\
&\geq &\frac{d(d-1)}{(d+1)^2}-7\a.
\end{eqnarray*}
Since $s\leq\frac{d}{d+1}+\a$, this yields
\begin{eqnarray*}
\left(\frac{d}{d+1}+\a\right)^2 p_{G[N]}(\ty)\geq  \frac{d(d-1)}{(d+1)^2}-7\a.
\end{eqnarray*}
By \eqref{alpha-definition1}, this yields
\begin{equation} \label{pg-lower2}
p_{G[N]}(\ty)\geq \frac{d-1}{d}-\delta_{d-1}(\frac{\ve}{2}).
\end{equation}
Since $d(G[N])\leq d-1$, by \eqref{pg-lower2} and induction hypothesis,
there exists $J\subseteq N$ with $|J|\leq d$ such that $\sum_{j\in J} y_j\geq 1-\frac{\ve}{2}$. Hence 
\begin{equation} \label{sum-lower1}
\sum_{j\in J} x_j\geq s\sum_{j\in J} y_j\geq s(1-\frac{\ve}{2})
\geq (\frac{d}{d+1}-\a)(1-\frac{\ve}{2})>\frac{d}{d+1}-\alpha-\frac{\ve}{2}.
\end{equation}

Let $I=J\cup \{1\}$. By \eqref{x1-lower2} and \eqref{sum-lower1},
we have $$\sum_{i\in I} x_i\geq \frac{1}{d+1}-\a+\frac{d}{d+1}-\a-\frac{\ve}{2}\geq 1-\ve.$$
This completes the induction step and the proof.
\end{proof}

In the next lemma, we use Lemma \ref{average-degree-stability} to
establish stability of the Lagrangian function of $T$-free graphs  where $T$ is a tree
that satisfies the Erd\H{o}s-S\'os conjecture. Note that such stability obviously does not exist for the Lagrangian function of  $K_k$-free graphs. To see that consider any complete $(k-1)$-partite graph $G$, which is $K_k$-free. Any  weight assignment $\tx$ on $G$ in which the total vertex weight on each part is $\frac{1}{k-1}$ satisfies $p_G(\tx)=\frac{k-2}{k-1}=\pi_\lambda(K_k)$, though any two such $\tx$'s can be very different.

\begin{lemma} \label{tree-lemma}
Let $k\geq 3$ be an integer. Let $T$ be a $k$-vertex tree that satisfies the Erd\H{o}s-S\'os conjecture. 
Let $0<\ve<1$ be a real. There exists a real $\delta_k(\ve)>0$ such that the following is true: If $G$ is a  $T$-free graph on $[t]$ and $\tx=(x_1,\ldots x_t)$ is a $1$-sum weight assignment on $G$ such that $p_G(\tx)\geq \frac{k-2}{k-1}-\delta_k$,
then $\exists I\subseteq [t]$ with $|I|\leq k-1$ such that $\sum_{i\in I} x_i\geq 1-\ve$.
\end{lemma}
\begin{proof}
Since $T$ satisfies the Erd\H{o}s-S\'os conjecture and $G$ is $T$-free, we have $d(G)\leq k-2$.
The lemma follows from Lemma \ref{average-degree-stability} with $d=k-2$.
\end{proof}

We can now use Lemma \ref{tree-lemma} to establish stability of the Lagrangian function of an enlarged tree.

\begin{lemma} \label{lagrangian-stable}
Let $k\geq 3, r\geq 2$ be integers where $k\geq M_r$.
Let $T$ be a $k$-vertex tree that satisfies the Erd\H{o}s-S\'os conjecture.
Let $F$ be the $(r-2)$-fold enlargement of $T$. 
Let $\ve>0$ be any real. There exists a real $\hd_r=\hd_r(\ve)>0$ such that following holds.
Let $G$ be a $F$-free $r$-graph on $[t]$
and $\tx=(x_1,\ldots, x_t)$ a $1$-sum weight assignment on $G$ such that $p_G(\tx)\geq \ckr-\hd_r$. Then there exists $I\subseteq [t]$ with $|I|\leq r+k-3$ such that $\sum_{i\in I} x_i\geq 1-\ve$. 
\end{lemma}
\begin{proof}
We use induction on $r$. The basis step $r=2$ was established by Lemma \ref{tree-lemma}. For the induction step, let $r\geq 3$. Let $\ve>0$ be given.
Let $\alpha$ be a real such that
$$0<\alpha<\frac{\ve}{4},$$ and
\begin{equation} \label{alpha-definition2}
r^{-1}\left(\frac{k+r-4}{k+r-3}+\a\right)^{-(r-1)}\cdot\left[r\frac{(k+r-4)_{r-1}}{(k+r-3)^{r-1}}-\a\right]>\frac{(k+r-4)_{r-1}}{(k+r-4)^{r-1}}-\hd_{r-1}(\frac{\ve}{2}). 
\end{equation}
Let $\hd_r=\gamma(\a)$, where the function $\gamma$ is given in Lemma 
\ref{representing-vertex}.  Suppose $\tx$ is a $1$-sum weight assignment on $G$ with
$p_G(\tx)\geq \ckr-\hd_r$.
By Lemma \ref{representing-vertex}, there exists $i\in [t]$, say $i=1$, such that
with $s=\sum_{j\in V(\cL_G(1)} x_j$ we have
\begin{equation} \label{x1-facts2}
|x_1-\frac{1}{k+r-3}|<\a, \quad |s-\frac{k+r-4}{k+r-3}|<\a, \quad \mbox { and } \lambda_1>r \frac{(k+r-3)_r}{(k+r-3)^r} - \a.
\end{equation}
For each $j\in V(\cL_G(1))$ let $y_j=\frac{x_j}{s}$. Then $\sum_{j\in V(\cL_G(1))} y_j=1$. Let $\ty$ denote the $1$-sum weight function
on $\cL_G(1)$ defined by the $y_j$'s. We have as usual
$$\lambda_1=r!\cdot \sum_{e\in \cL_G(1)} \prod_{j\in e} x_j = r s^{r-1} (r-1)!\cdot \sum_{e\in \cL_G(1)}\prod_{j\in e} y_j\leq rs^{r-1}p_{\cL_G(1)}(\ty).$$
Hence, by \eqref{x1-facts2}, we have
$$r\left(\frac{k+r-4}{k+r-3}+\a\right)^{r-1} p_{\cL_G(1)}(\ty)>r\frac{(k+r-3)_r}{(k+r-3)^r}-\a=r\frac{(k+r-4)_{r-1}}{(k+r-3)^{r-1}}-\a.$$

By \eqref{alpha-definition2}, this yields
\begin{equation}\label{pg-lower3}
p_{\cL_G(1)}(\ty)>\frac{(k+r-4)_{r-1}}{(k+r-4)^{r-1}}-\hd_{r-1}(\frac{\ve}{2})=(k-2)f_{r-1}(k)-\hd_{r-1}(\frac{\ve}{2}).
\end{equation}
Let $F'$ denote the $(r-3)$-enlargement of $T$. Since $G$ is $T$-free, clearly
$\cL_G(1)$ is $T'$-free. Since $F'$ is the $(r-3)$-enlargement of $T$, where
$T$ is $k$-vertex tree satisfying the Erd\H{o}s-S\'os conjecture, and $k\geq M_r\geq M_{r-1}$, by \eqref{pg-lower3} and induction hypothesis,
there exists $J\subseteq V(\cL_G(1))$ such that $\sum_{j\in J} y_j\geq 1-\frac{\ve}{2}$.
Hence 
\begin{equation} \label{sum-lower2}
\sum_{j\in J} x_j\geq s(1-\frac{\ve}{2})\geq (\frac{k+r-4}{k+r-3}-\a)(1-\frac{\ve}{2})\geq \frac{k+r-4}{k+r-3}-\a-\frac{\ve}{2}.
\end{equation}
Let $I=J\cup \{1\}$. We have  by \eqref{x1-facts2} and \eqref{sum-lower2} 
$$\sum_{i\in I}x_i\geq \frac{1}{k+r-3}-\a+\frac{k+r-4}{k+r-3}-\a-\frac{\ve}{2}\geq 1-\ve.$$
This completes the induction and the proof.
\end{proof}

{\bf Proof of Theorem \ref{enlarged-tree-stable}:}
Let $\ve>0$ be given.  We may assume $\ve$ to be sufficiently small so that
$\ve<\gamma_0$, where $\gamma_0$ is defined in Theorem \ref{stability-general}.
Let $\gamma=\min\{\frac{1}{2}\hd_r(\frac{\ve}{2}),\frac{\ve}{2}\}$, where $\hd_r$ is given in Lemma \ref{lagrangian-stable}.  By our definition, $\gamma<\gamma_0$.
Let $\delta$ and $n_0$ be the constants
guaranteed by Theorem \ref{stability-general} for the above defined $\gamma$.
 Let $n$ be sufficiently large so that 
$n\geq n_0$ and that $n$ satisfies some other inequalities given below.
Let $G$ be any $\cH_{k+r-2}^F$-free $r$-graph on $[n]$ with
$|G|>(\ckr-\delta)\binom{n}{r}$.
Let $G^*$ be the final graph produced by Algorithm \ref{symm-clean} with
threshold $\ckr-\gamma$. By Theorem \ref{stability},
$G^*$ is $\cH_{k+r-2}^F$-free, $n(G^*)\geq (1-\gamma)n$, and
$G^*$ is $(\ckr-\gamma)$-dense. Let $N=n(G^*)$. Since $G^*$ is 
$(\ckr-\gamma)$-dense, we have 
\begin{equation} \label{G*-lower2}
|G^*|\geq (\ckr-\gamma)\binom{N}{r}. 
\end{equation}
Suppose $G$ has $s$ equivalence classes
$A_1,\ldots, A_s$. Let $S$ consist of one vertex from each equivalence class of $G^*$.  Without loss of generality, suppose $S=[s]$.
 Then $G^*[S]$ covers pairs and $G^*$ is a blowup of
$G^*[S]$. For each $i\in [s]$, let $x_i=|A_i|/N$.
Then $\sum_i x_i=1$.
So $\tx=(x_1,\ldots, x_s)$ is a $1$-sum weight assignment on $G^*[S]$. 
Also,$$p_G(\tx)=r!\sum_{e\in G}\prod_{i\in e} x_i =\frac{r!}{N^r} |G^*|\geq 
\ckr-2\gamma\geq\ckr-\hd_r(\frac{\ve}{2}),$$
where the two inequalities follows from \eqref{G*-lower2}, our definition of $\gamma$, and the assumption that $N$ is sufficiently large.
By Lemma \ref{lagrangian-stable}, there exists $I\subseteq [s]$, where $|I|\leq k+r-3$ such
that $\sum_{i\in I} x_i\geq 1-\frac{\ve}{2}$. Let $W=\bigcup_{i\in I} A_i$.
Then $$|W|\geq (1-\frac{\ve}{2})N\geq (1-\gamma_0)N.$$
By Theorem \ref{stability-general}, $G[W]$ is $(k+r-3)$-partite.
Since $$|W|\geq (1-\frac{\ve}{2})N\geq (1-\frac{\ve}{2})(1-\gamma)n\geq
(1-\frac{\ve}{2})^2 n>n-\ve n,$$
$G$ can be made $(k+r-3)$-partite by deleting at most $\ve n$ vertices. So, $\cH_{k+r-2}^F$ is $(k+r-3)$-stable.
\qed


\end{document}